\documentclass[12pt]{amsart}
\usepackage{amsmath,amssymb,amsthm,latexsym,amsbsy,amsfonts}
\usepackage{color,verbatim}
\usepackage{pstricks}
\usepackage{tabularx}
\usepackage{rotating}
\usepackage{graphicx,epsfig}
\newcommand{\CX}{\C_{\times}^\infty(\wtd D)}
\newcommand{\CTE}[1]{C_T^1(0,T;{#1})}
\newcommand{\CoE}[1]{C_0^1(0,T;{#1})}
\newcommand{\C}{{\mathbb C}}
\newcommand{\R}{{\mathbb R}}
\newcommand{\D}{{\mathbb D}}
\newcommand{\E}{{\mathbb E}}

\newcommand{\mP}{\mathbb P}

\newcommand{\mE}{\mathbb E}
\newcommand{\mH}{\mathbb H}
\newcommand{\mL}{\mathbb L}
\newcommand{\mT}{\mathbb T}
\newcommand{\mV}{\mathbb V}
\newcommand{\mW}{\mathbb W}
\newcommand{\mY}{\mathbb Y}
\newcommand{\inpro}[2]{\left\langle{#1},{#2}\right\rangle}

\newcommand{\norm}[2]{\|{#1}\|_{#2}}

\newcommand{\snorm}[2]{\left|{#1}\right|_{#2}}

\newcommand{\lam}{\lambda}

\newcommand{\sig}{\sigma}

\newcommand{\vecH}{\boldsymbol{H}}

\newcommand{\vecM}{\boldsymbol{M}}

\newcommand{\vecP}{\boldsymbol{P}}

\newcommand{\veca}{\boldsymbol{a}}
\newcommand{\vecb}{\boldsymbol{b}}
\newcommand{\vecc}{\boldsymbol{c}}

\newcommand{\vece}{\boldsymbol{e}}

\newcommand{\vecg}{\boldsymbol{g}}

\newcommand{\vecm}{\boldsymbol{m}}
\newcommand{\vecn}{\boldsymbol{n}}

\newcommand{\vecu}{\boldsymbol{u}}
\newcommand{\vecv}{\boldsymbol{v}}
\newcommand{\vecw}{\boldsymbol{w}}
\newcommand{\vecx}{\boldsymbol{x}}

\newcommand{\vectau}{\boldsymbol{\tau}}
\newcommand{\vecvarphi}{\boldsymbol{\varphi}}

\newcommand{\vecpsi}{\boldsymbol{\psi}}
\newcommand{\vecxi}{\boldsymbol{\xi}}
\newcommand{\veczeta}{\boldsymbol{\zeta}}

\DeclareMathOperator{\dive}{{div \/}}

\DeclareMathOperator{\curl}{{curl \/}}

\DeclareMathOperator*{\esssup}{ess\,sup \/}

\newcommand{\cD}{{\mathcal D}}
\newcommand{\cE}{{\mathcal E}}
\newcommand{\cF}{{\mathcal F}}

\newcommand{\cI}{{\mathcal I}}

\newcommand{\cL}{{\mathcal L}}
\newcommand{\cM}{{\mathcal M}}
\newcommand{\cN}{{\mathcal N}}

\newcommand{\goto}{\rightarrow}
\newcommand{\gotoo}{\longrightarrow}

\newcommand{\p}{\partial}
\newcommand{\wtd}{\widetilde}

\newcommand{\ol}{\overline}

\newcommand{\pa}{\partial}

\newcommand{\ds}{\, ds}
\newcommand{\dt}{\, dt}

\newcommand{\dW}{\, dW}

\newcommand{\dvx}{\, d\vecx}

\newcommand{\nn}{\nonumber}

\numberwithin{equation}{section}
\newtheorem{theorem}{Theorem}[section]
\newtheorem{lemma}[theorem]{Lemma}

\newtheorem{proposition}[theorem]{Proposition}

\newtheorem{definition}[theorem]{Definition}

\title[FEM for stochastic MLLG equation]{A FINITE ELEMENT APPROXIMATION FOR 
THE STOCHASTIC Maxwell--LANDAU--LIFSHITZ--GILBERT SYSTEM}

\author{Beniamin Goldys}
\address{School of Mathematics and Statistics,
         The University of Sydney,
         Sydney 2006, Australia}
\email{beniamin.goldys@sydney.edu.au}
\author{Kim-Ngan Le}
\address{School of Mathematics and Statistics,
         The University of New South Wales,
         Sydney 2052, Australia}
\email{n.le-kim@unsw.edu.au}
\author{Thanh Tran}
\address{School of Mathematics and Statistics,
         The University of New South Wales,
         Sydney 2052, Australia}
\email{thanh.tran@unsw.edu.au}

\date{\today}
\newtheorem{algorithm}{Algorithm}[section]

\begin{document}
\begin{abstract}
The stochastic Landau--Lifshitz--Gilbert (LLG) equation coupled with the Maxwell
equations (the so called stochastic MLLG system) describes the creation of
domain walls and vortices (fundamental objects for the novel nanostructured
magnetic memories).
We first reformulate the stochastic LLG equation into an equation 
with time-differentiable solutions.
We then propose a convergent $\theta$-linear scheme to approximate the
solutions
of the reformulated system. As a consequence, we prove convergence of the
approximate solutions, with no or minor conditions on time and space steps
(depending on the value of $\theta$).
Hence, we prove the existence of weak martingale solutions of the stochastic 
MLLG system. Numerical results are presented to show applicability of the method.
\end{abstract}
\maketitle
\section{Introduction}
The Maxwell--Landau--Lifshitz--Gilbert (MLLG) system
 describes the electromagnetic behaviour of a 
ferromagnetic material~\cite{Cimrak_survey}. 
For simplicity, we suppose that there is a bounded cavity $\wtd D\subset\R^3$
(with perfectly conducting outer surface $\p\wtd D$) in which a ferromagnet
$D$ is embedded, and $\wtd D\backslash \bar{D}$ is an
isotropic material. 
Letting $D_T:= (0,T)\times D$ and $\wtd D_T:= (0,T)\times \wtd D$, the
magnetisation field $\vecM : D_T\goto \mathbb{S}^2$ (where $\mathbb{S}^2$ is
the unit sphere in $\R^3$) and the magnetic field $\vecH:\wtd D_T\goto
\mathbb{R}^3$  satisfy the quasi-static model of  the MLLG system:
\begin{align}
&\vecM_t
=
\lambda_1\vecM\times\vecH_{\text{eff}}
-
\lambda_2\vecM\times(\vecM\times\vecH_{\text{eff}})
\quad\text{ in } D_T,\label{E:LL}\\
&\mu_0\vecH_t
+
\nabla\times(\sig\nabla\times\vecH) 
=
-\mu_0\wtd\vecM_t
\quad\text{ in } \wtd D_T,\label{E:Max2}
\end{align}
in which $\lambda_1\not=0$, $\lambda_2 >0$,  and $\mu_0> 0$ are
constants. Here, the inverse of the conductivity $\sigma$ is a scalar positive bounded function on $\wtd D$ satisfying
$\sigma(\vecx) = \sig_D>0$ for all $\vecx\in D$~\cite{Monk03}.
Vector function $\vecH_{\text{eff}}$ is the effective field  and
$\wtd\vecM:\wtd D_T\goto \R^3$ is the zero extension of $\vecM$ onto $\wtd D_T$, i.e.,
\begin{equation*}
\wtd\vecM(t,\vecx)
=
\begin{cases}
\vecM(t,\vecx),&\quad (t,\vecx)\in D_T,
\\
\hfill 0, &\quad (t,\vecx)\in \wtd D_T\setminus\ol{D}_T.
\end{cases}
\end{equation*}
The system \eqref{E:LL}--\eqref{E:Max2} is supplemented with the initial conditions
\begin{equation}\label{InCond}
\vecM(0,\cdot)=\vecM_0\text{ in }D\quad\text{and}\quad
\vecH(0,\cdot)=\vecH_0\text{ in }\wtd D,
\end{equation}
and the boundary conditions
\begin{equation}\label{Cond}
\partial_{\vecn_D}\vecM=0 
\text{ on } (0,T)\times\partial D\quad\text{and}\quad
(\nabla\times\vecH)\times \vecn_{\wtd D} =0 
\text{ on } (0,T)\times\partial \wtd D,
\end{equation}
where $\vecn_D$ and $\vecn_{\wtd D}$ are the unit outward normal vectors to 
$D$ and $\wtd D$, respectively. Here $\partial_{\vecn_D}$ denotes the normal
derivative.

It is highly significant to consider the stochastic MLLG system in order to
describe the creation of domain walls and vortices (fundamental objects for
the novel nanostructured magnetic memories)~\cite{Moser04}. 
We follow~\cite{BanBrzPro09,BrzGolJer12} to add a noise to the effective
field~$\vecH_{\text{eff}}$ so that the
stochastic version of the MLLG system takes the form
\begin{align}
&d\vecM
=
\big(\lambda_1 \vecM\times \vecH_{\text{eff}}
-
\lambda_2 \vecM\times(\vecM\times\vecH_{\text{eff}})\big)\dt
+
(\vecM\times \vecg)\circ dW(t)
\,\text{ in } D_T,\label{E:1.1}\\
&\mu_0 \,d\vecH
+
\nabla\times(\sig\nabla\times\vecH) \dt
=
-\mu_0 \, d\wtd\vecM
\,\text{ in } \wtd D_T,\label{E:Maxwell2}
\end{align}
where $\vecg : D\goto\R^3$ is a given bounded function, and $W$ is a 
one-dimensional Wiener process.
Here $\circ dW(t)$ stands for the Stratonovich
differential.
We assume without loss of generality that
(see~\cite{BrzGolJer12})
\begin{equation}\label{equ:g 1}
\snorm{\vecg(\vecx)}{}
= 1, \quad \vecx\in D
\end{equation}
For simplicity the effective field $\vecH_{\text{eff}}$ is taken 
to be $\vecH_{\text{eff}}=\Delta\vecM+\vecH$.

In the deterministic case, i.e.~\eqref{E:LL}--\eqref{E:Max2}, 
the existence and uniqueness of a {\em local} strong solution 
is shown by Cimr{\'a}k~\cite{CimExist07}. He also
proposes~\cite{CimError07} a finite element method to
approximate this local solution and provides error
estimation. Various results on the existence of global weak
solutions are proved in \cite{GuoSu97, GuoSu01, SuGuo98}. A more complete
list can be found in~\cite{Cimrak_survey,GuoDing08,KruzikProhl06}.
It should be noted that apart from \cite{CimError07} where a
numerical scheme is suggested for a local solution,
other analyses are non-constructive, namely no
computational techniques are proposed for the solution. 

In~\cite{MonkVacus01},
the stability of a semidiscrete scheme to numerically 
solve~\eqref{E:LL}--\eqref{E:Max2} is
verified, but its convergence is not studied. 
Ba\v{n}as, Bartels and Prohl~\cite{BBP08}
propose an implicit {\em nonlinear} scheme to 
solve the MLLG system, and succeed in proving
that the finite element solution converges to a weak
{\em global} solution of the problem. 
A $\theta$-{\em linear} finite
element scheme is proposed
in~\cite{BanPagPra12,LMDT13,LeTra12}
to find a weak {\em global} solution 
to the MLLG system, and convergence of  the numerical 
solutions is proved with no 
condition imposed on time step and space step if $\theta\in(\frac{1}{2},1]$. 
It should be mentioned that the proofs of existence proposed
in~\cite{BBP08,BanPagPra12,LMDT13,LeTra12} are
constructive proofs, namely an approximate solution can be
computed.

In the stochastic case,
the Faedo--Galerkin method is used
in \cite{BrzGolJer12} to show the existence of
a weak martingale solution for the stochastic Landau--Lifshitz--Gilbert (LLG) equation~\eqref{E:1.1}.
Finite element schemes for this equation are studied
in~\cite{Alo14,BanBrzPro09,BNT13} which prove that the numerical
solutions converge to a weak martingale solution. 
It is noted that 
a {\em non-linear} scheme is proposed in~\cite{BanBrzPro09} and 
{\em linear} schemes are proposed in~\cite{Alo14,BNT13}. 

The full version of the stochastic Landau--Lifshitz equation coupled with the Maxwell's equations 
is studied firstly in~\cite[Section 5]{LiangThesis} where the existence of the weak martingale solution and its regularity 
are proved by using the Faedo-Galerkin approximation, the methods of compactness and Skorokhod's Theorem. 

To the best of our knowledge the numerical analysis of the system~\eqref{E:1.1}--\eqref{E:Maxwell2} is an open problem at present.
In this paper, 
we extend the {\em $\theta$-linear} finite element scheme developed
in~\cite{LeTra12} for the deterministic MLLG system to the stochastic case. 
Since this scheme seeks to approximate the time derivative of the
magnetization $\vecM$, we adopt the technique in~\cite{BNT13} 
to reformulate system~\eqref{E:1.1}--\eqref{E:Maxwell2} into a system not
involving the Stratonovich differential~$\circ dW(t)$. 
Then the {\em $\theta$-linear} scheme mentioned above can
be applied. As a consequence,
we prove  the existence of weak martingale solutions to the stochastic MLLG
system.

The paper is organised as follows.
In Section~\ref{sec:not} we define the notations to be used, and recall some
technical results. In Section~\ref{sec:mai res} we
define weak martingale solutions to~\eqref{E:1.1}--\eqref{E:Maxwell2}
and state our main result.
Details of the reformulation of~\eqref{E:1.1} are presented in
Section~\ref{sec:equ eqn}. We also show in this section how
a weak solution to~\eqref{E:1.1}--\eqref{E:Maxwell2} can be obtained from a
weak solution of the reformulated system.
In Section~\ref{sec:fin ele}, we introduce our finite element
scheme and present a proof of the convergence of finite
element solutions to a weak solution of the reformulated
system.
Section~\ref{sec:pro} is devoted to the proof of
the main theorem. Our numerical experiments are
presented in Section~\ref{sec:num}.

Throughout this paper, $c$ denotes a generic constant which may take 
different values at different occurences.  
\section{Notations and technical results}\label{sec:not}
\subsection{Notations}
In this subsection, we introduce some function spaces and notations which are
used in the rest of this paper.

For any open set $U\subset\R^3$, the {\em curl} operator of a vector function
$\vecu=(u_1,u_2,u_3)$ defined on $U$ is denoted by
\[
\text{curl}\, \vecu
=
\nabla\times\vecu
:=
\big(
\frac{\partial u_2}{\partial x_3} - \frac{\partial u_3}{\partial x_2},
\frac{\partial u_1}{\partial x_3} - \frac{\partial u_3}{\partial x_1},
\frac{\partial u_2}{\partial x_1} - \frac{\partial u_1}{\partial x_2}
\big),
\]
if the partial derivatives exist.
The function spaces $\mH^1(U)$ and $\mH(\curl;U)$ are defined, respectively,
by
\begin{gather*}
\mH^1(U)
=
\left\{ \vecu\in\mL^2(U) : \frac{\p\vecu}{\p
x_i}\in \mL^2(U)\quad\text{for } i=1,2,3
\right\},\\
\mH(\curl; U)
=
\left\{
 \vecu\in\mL^2(U) : \nabla\times\vecu\in \mL^2(U)
\right\}.
\end{gather*}
Here, $\mL^2(U)$ is the usual space of Lebesgue square
integrable functions defined on $U$ and taking values in $\R^3$.
The inner product and norm in $\mL^2(U)$ are
denoted by $\inpro{\cdot}{\cdot}_{U}$ and
$\|\cdot\|_{U}$, respectively.

For any vector functions $\vecu, \vecv, \vecw$,
we denote
\begin{equation}\label{equ:nab nab}
\begin{gathered}
\nabla\vecu\cdot\nabla\vecv
:=
\sum_{i=1}^3
\frac{\partial\vecu}{\partial x_i}
\cdot
\frac{\partial\vecv}{\partial x_i}, 
\\
\nabla\vecu\times\nabla\vecv
:=
\sum_{i=1}^3
\frac{\partial\vecu}{\partial x_i}
\times
\frac{\partial\vecv}{\partial x_i}, 
\\
\vecu\times\nabla\vecv
:=
\sum_{i=1}^3
\vecu
\times
\frac{\partial\vecv}{\partial x_i}, 
\\
\left(\vecu\times\nabla\vecv\right) \cdot \nabla\vecw
:=
\sum_{i=1}^3
\left(\vecu\times\frac{\partial\vecv}{\partial x_i}\right)
\cdot
\frac{\partial\vecw}{\partial x_i},
\end{gathered}
\end{equation}
provided that the partial derivatives exist, at least in the weak sense.
We also denote 
\begin{align*}
\C^\infty(\wtd D)
&:=
\left\{
\vecu : \wtd D \to \R^3 \ | \ \vecu \text{ is infinitely differentiable}
\right\},
\\
\CX 
&:=
\left\{\vecu\in 
\mathbb C^{\infty}(\wtd D)\cap \C(\bar D)
\ | \
(\nabla\times\vecu)\times\vecn_D = 0
\text{ on } 
\partial D
\text{ and }
(\nabla\times\vecu)\times\vecn_{\wtd D} = 0
\text{ on } 
\partial\wtd D
\right\},
\\
\CTE{E}
&:=
\left\{
\vecu : [0,T] \to E \ | \ \vecu \text{ is continuously differentiable and } 
\vecu(T) = 0 \text{ in } E
\right\},
\\
\CoE{E}
&:=
\left\{
\vecu : [0,T] \to E \ | \ \vecu \text{ is continuously differentiable and } 
\vecu(0) = \vecu(T) = 0 \text{ in } E
\right\},
\end{align*}
for any $T>0$ and any normed vector space $E$.

\subsection{Technical results}
In this subsection we recall some results from~\cite{BNT13}. They will
be used in the next section to reformulate~\eqref{wE:1.1} to a new form.

Assume that $\vecg\in\mL^{\infty}(D)$, and let
$G : \mL^2(D) \gotoo \mL^2(D)$ be defined by
\begin{equation}\label{equ:G}
G\vecu = \vecu\times\vecg \quad\forall\vecu\in\mL^2(D).
\end{equation}
Then the operator $G$ is bounded~\cite{BNT13}.
\begin{lemma}\label{lem:G2}
For any $s\in\R$ and $\vecu,\vecv\in\mL^2(D)$ there hold
\begin{align}
e^{sG}\vecu
&=
\vecu + (\sin s) G\vecu + (1-\cos s) G^2\vecu,
\label{equ:G8} \\
\left(e^{sG}\right)^*
&=
e^{-sG},
\label{equ:eadj}\\
e^{sG}G\vecu &= Ge^{sG}\vecu,
\label{equ:G9} \\
e^{sG}(\vecu\times\vecv)
&=
e^{sG}\vecu \times e^{sG}\vecv. 
\label{equ:G11}
\end{align}
\end{lemma}
In the proof of the existence of weak solutions we also need the
following result for the operator~$e^{sG}$.

\begin{lemma}\label{lem:eC wea}
Assume that $\vecg\in\mH^2(D)$. For any $s\in\R$,
$\vecu\in\mH^1(D)$ and $\vecv\in\mW_0^{1,\infty}(D)$, let 
 \[
 \wtd C(s,\vecv)
 =
  e^{-sG}
 \big(
 (\sin s) C + (1-\cos s)(GC+CG)
 \big)\vecv
 \]
with $C$ being defined by
\[
C\vecu
=
\vecu\times\Delta\vecg
+ 2\sum_{i=1}^3
\frac{\pa\vecu}{\pa x_i}
\times
\frac{\pa\vecg}{\pa x_i}.
\]
There holds
\begin{equation*}
\inpro{\wtd C(s,e^{-sG}\vecu)}{\vecv}_D
=
\inpro{\nabla\left(e^{-sG}\vecu\right)}{\nabla\vecv}_D
-
\inpro{\nabla \vecu}{\nabla\left(e^{sG}\vecv\right)}_D,
\end{equation*}

\end{lemma}

From now on, we assume that $\vecg\in\mW^{2,\infty}(D)$.

We finish this section by stating two elementary identities involving the dot
and cross products of vectors in $\R^3$, which will be frequently used. 
For all $\veca, \vecb, \vecc\in\R^3$, the following identities hold
\begin{equation}\label{equ:abc}
\veca\times(\vecb\times\vecc) =
(\veca\cdot\vecc)\vecb
-
(\veca\cdot\vecb)\vecc
\end{equation}
and
\begin{equation}\label{equ:ele ide}
(\veca\times\vecb)\cdot\vecc
=
(\vecb\times\vecc)\cdot\veca
=
(\vecc\times\veca)\cdot\vecb.
\end{equation}

\section{The main result}\label{sec:mai res}
In this section we state the definition of a weak martingale
solution to~\eqref{E:1.1}--\eqref{E:Maxwell2} and our main result.

Recalling that $\vecH_{\text{eff}} = \Delta\vecM + \vecH$,
multiplying~\eqref{E:1.1} by a test function~$\vecpsi\in\C^{\infty}(D)$ and integrating over
$(0,t)\times D$ we obtain formally
\begin{align*}
\inpro{\vecM(t)}{\vecpsi}_D
-
\inpro{\vecM_0}{\vecpsi}_D
&=
\lambda_1 \int_0^t
\inpro{\vecM\times\Delta\vecM}{\vecpsi}_D \ds
+
\lambda_1 \int_0^t
\inpro{\vecM\times\vecH}{\vecpsi}_D \ds
\\
&\quad
-
\lambda_2 \int_0^t
\inpro{\vecM\times(\vecM\times\Delta\vecM)}{\vecpsi}_D \ds
\\
&\quad
-
\lambda_2 \int_0^t
\inpro{\vecM\times(\vecM\times\vecH)}{\vecpsi}_D \ds
\\
&\quad
+
\int_0^t
\inpro{\vecM\times\vecg}{\vecpsi}_D \circ dW.
\end{align*}
From~\eqref{equ:ele ide}, the Green identity and
$\nabla\vecM \cdot (\nabla\vecM\times\vecpsi) = 0$ we define
\begin{align*}
\inpro{\vecM\times\Delta\vecM}{\vecpsi}_D
&=
- \inpro{\Delta\vecM}{\vecM\times\vecpsi}_D\\
&:=
\inpro{\nabla\vecM}{\nabla(\vecM\times\vecpsi)}_D
\\
&
=
\inpro{\nabla\vecM}{\nabla\vecM\times\vecpsi}_D
+
\inpro{\nabla\vecM}{\vecM\times\nabla\vecpsi}_D
\\
&
=
-
\inpro{\vecM\times\nabla\vecM}{\nabla\vecpsi}_D,
\end{align*}
and similarly
\[
\inpro{\vecM\times(\vecM\times\Delta\vecM)}{\vecpsi}_D
:=
\inpro{\vecM\times\nabla\vecM}{\nabla(\vecM\times\vecpsi)}_D.
\]
Therefore,
\begin{align*}
\inpro{\vecM(t)}{\vecpsi}_{D}
-
\inpro{\vecM_0}{\vecpsi}_{D}
&=
-\lambda_1
\int_0^t
\inpro{\vecM\times\nabla\vecM}{\nabla\vecpsi}_{D}\ds
+
\lambda_1 \int_0^t
\inpro{\vecM\times\vecH}{\vecpsi}_{D}\ds
\nn\\
&\quad
-
\lambda_2
\int_0^t
\inpro{\vecM\times\nabla\vecM}{\nabla(\vecM\times\vecpsi)}_{D}\ds
\nn\\
&\quad 
-
\lambda_2 \int_0^t
\inpro{\vecM\times(\vecM\times\vecH)}{\vecpsi}_{D}\ds
+
\int_0^t
\inpro{\vecM\times\vecg}{\vecpsi}_{D}\circ dW(s).
\end{align*}

In the same manner, if we multiply~\eqref{E:Maxwell2} by a test function
$\veczeta\in\CTE{\CX}$, integrate over~$\wtd D_T$, and note~\eqref{InCond},
then we obtain, formally,
\[
\mu_0
\inpro{\vecH+\wtd\vecM}{\veczeta_t}_{\wtd D_T}
-
\mu_0
\inpro{\vecH_0+\wtd\vecM_0}{\veczeta(0)}_{\wtd D}
=
\inpro{\nabla\times(\sigma\nabla\times\vecH)}{\veczeta}_{\wtd D_T}.
\]
We remark that the time derivative is taken on $\veczeta$ because in
general $\wtd\vecM$ is not time differentiable.
Since 
$(\nabla\times\vecH) \times \vecn_{\wtd D} = 0$, see~\eqref{Cond}, and 
$(\nabla\times\veczeta) \times \vecn_{\wtd D} = 0$, see the definition
of~$\CX$ in Section~\ref{sec:not}, it follows 
from~\cite[Corollary 3.20]{Monk03} that 
\[
\inpro{\nabla\times(\sigma\nabla\times\vecH)}{\veczeta}_{\wtd D}
=
\inpro{\sigma\nabla\times\vecH}{\nabla\times\veczeta}_{\wtd D}.
\]
Hence
\[
\mu_0
\inpro{\vecH+\wtd\vecM}{\veczeta_t}_{\wtd D_T}
-
\mu_0
\inpro{\vecH_0+\wtd\vecM_0}{\veczeta(0)}_{\wtd D}
=
\inpro{\sigma\nabla\times\vecH}{\nabla\times\veczeta}_{\wtd D_T}.
\]

The above observations prompt us to define the solution
of~\eqref{E:1.1}--\eqref{E:Maxwell2} as follows.
\begin{definition}\label{def:wea sol}
Given $T\in(0,\infty)$, a weak martingale solution
$(\Omega,\cF,(\cF_t)_{t\in[0,T]},\mP,W,\vecM,\vecH)$
to~\eqref{E:1.1}--\eqref{E:Maxwell2} on the time interval $[0,T]$,
consists of 
\begin{enumerate}
\renewcommand{\labelenumi}{(\alph{enumi})}
\item
a filtered probability space
$(\Omega,\cF,(\cF_t)_{t\in[0,T]},\mP)$ with the
filtration satisfying the usual conditions,
\item
a one-dimensional $(\cF_t)$-adapted Wiener process
$W=(W_t)_{t\in[0,T]}$,
\item 
a progressively measurable
process $\vecM : [0,T]\times\Omega \goto \mL^2(D)$,
\item 
a progressively measurable
process $\vecH : [0,T]\times\Omega \goto \mL^2(\wtd D)$
\end{enumerate}
such that there hold
\begin{enumerate}
\item
\quad
$\mP\big(\vecM \in C([0,T];\mH^{-1}(D))\big)=1$;
\item
\quad
$\mP\big(\vecH \in L^2(0,T;\mH(\curl;\wtd D)\big)=1$;
\item
\quad
$\mE\left( 
\esssup_{t\in[0,T]}\|\nabla\vecM(t)\|^2_D 
\right) < \infty$;
\item
\quad for all $t \in [0,T]$,
$|\vecM(t,\cdot)| = 1$  a.e.~in $D$, and $\mP$-a.s.;
\item
\quad
for every  $t\in[0,T]$,
for all $\vecpsi\in\C^{\infty}(D)$, 
$\mP$-a.s.:
\begin{align}\label{wE:1.1}
\inpro{\vecM(t)}{\vecpsi}_{D}
-
\inpro{\vecM_0}{\vecpsi}_{D}
&=
-\lambda_1
\int_0^t
\inpro{\vecM\times\nabla\vecM}{\nabla\vecpsi}_{D}\ds\nn\\
&\quad-
\lambda_2
\int_0^t
\inpro{\vecM\times\nabla\vecM}{\nabla(\vecM\times\vecpsi)}_{D}\ds\nn\\
&\quad 
+
\lambda_1 \int_0^t
\inpro{\vecM\times\vecH}{\vecpsi}_{D}\ds
\nn\\
&\quad 
-
\lambda_2 \int_0^t
\inpro{\vecM\times(\vecM\times\vecH)}{\vecpsi}_{D}\ds
\nn\\
&\quad+
\int_0^t
\inpro{\vecM\times\vecg}{\vecpsi}_{D}\circ dW(s);
\end{align}
\item
\quad
for all $\veczeta\in\CTE{\CX}$, $\mP$-a.s.:
\begin{equation}\label{wE:Maxwell2}
\mu_0\inpro{\vecH+\wtd\vecM}{\veczeta_t}_{\wtd D_T}
-
\mu_0\inpro{\vecH_0+\wtd\vecM_0}{\veczeta(0,\cdot)}_{\wtd D}
=
\inpro{\sig\nabla\times\vecH}{\nabla\times\veczeta}_{\wtd D_T}.
\end{equation}
\end{enumerate}
\end{definition}

The main theorem of the paper is stated below.
\begin{theorem}\label{the:mai}
Assume that $\vecg\in\mW^{2,\infty}(D)$
satisfies~\eqref{equ:g 1} and $\big(\vecM_0,\vecH_0\big)$
satisfies
\begin{equation}\label{E:Cond1}
\begin{aligned}
\vecM_0\in\mH^2(D)
&,\quad
|\vecM_0|=1\text{\ a.e. in }D, 
\\
(\vecH_0+\wtd\vecM_0)\in\mH^1(\wtd D)
&,\quad 
\nabla\times(\vecH_0+\wtd\vecM_0) \in\mH^1(\wtd D).
\end{aligned}
\end{equation}
For each $T>0$, there exists a weak
martingale solution to~\eqref{E:1.1}--\eqref{E:Maxwell2}.
\end{theorem}
\begin{proof}
The theorem is a direct consequence of Theorem~\ref{the:mai 2}.
\end{proof}

\section{Equivalence of weak solutions}\label{sec:equ eqn}
In this section, we use the operator $G$ defined in Section~\ref{sec:not} to 
define new variables~$\vecm$ and $\vecP$ from $\vecM$ and $\vecH$.

Informally, if $\big(\vecM,\vecH\big)$ is a weak solution
to~\eqref{wE:1.1}--\eqref{wE:Maxwell2} then we can define new 
processes~$\vecm$ and $\vecP$ (see~\eqref{equ:vecm}--\eqref{equ:vecP} below) 
such that the Stratonovich differential $\circ dW(t)$ vanishes in the
partial differential equation satisfied by $\vecm$.
Moreover, it will be seen that $\vecm$ is differentiable with
respect to $t$. We will make this argument more rigorous in the following lemma.

Let a filtered probability space $\big(\Omega,\cF,(\cF_t)_{t\in[0.T]},\mP\big)$ 
and a Wiener process $W(t)$ on it be given.
We define a new processes $\vecm$ and $\vecP$ from processes $\vecM$ and $\vecH$ 
\begin{align}
\vecm(t,\cdot) 
&:=
e^{-W(t)G}\vecM(t,\cdot) \quad\forall t
\ge0, \ a.e. \,\text{ in } D,\label{equ:vecm}\\
\vecP(t,\cdot)
&:=
\vecH(t,\cdot) + \wtd\vecM(t,\cdot)
\quad\forall t
\ge0, \ a.e. \,\text{ in } \wtd D,\label{equ:vecP}
\\
\vecP_0
&:=
\vecH_0 + \wtd\vecM_0
\quad a.e. \, \text{ in } \wtd D,
\nn
\end{align}
where $\wtd\vecM_0$ is the zero extension of~$\vecM_0$ onto~$\wtd D$.
Then it follows immediately from~\eqref{equ:G12} and~\eqref{equ:eadj} that,
for all $t\in[0,T]$ and almost all $x\in D$, 
\begin{equation}\label{equ:m M 1}
\snorm{\vecM(t,\cdot)}{} = 1
\quad\text{if and only if}\quad 
\snorm{\vecm(t,\cdot)}{} = 1.
\end{equation}
The following lemma shows that in order to find~$\vecM$ and~$\vecH$, it
suffices to find~$\vecm$ and~$\vecP$.
\begin{lemma}\label{lem:we mP}
Let $\vecm\in H^1\big(0,T;\mH^1(D)\big)$ and 
$\vecP\in L^2(0,T;\mH(\curl;\wtd D))$, $\mP$-a.s., satisfy
\begin{multline}\label{InE:14}
\inpro{\vecm_t}{\vecxi}_{D_T}
+
\lambda_1
\inpro{\vecm\times\nabla\vecm}{\nabla\vecxi}_{D_T}
+
\lambda_2
\inpro{\vecm\times\nabla\vecm}{\nabla(\vecm\times\vecxi)}_{D_T}
-
\inpro{F(t,\vecm)}{\vecxi}_{D_T}
\\
-
\lambda_1\inpro{\vecm\times e^{-W(t)G}\vecP}{\vecxi}_{D_T}
+
\lambda_2\inpro{\vecm\times(\vecm\times e^{-W(t)G}\vecP)}{\vecxi}_{D_T}
=0 
\end{multline}
and
\begin{equation}\label{wE:vecP}
\mu_0\inpro{\vecP}{\veczeta_t}_{\wtd D_T}
-
\mu_0\inpro{\vecP_0}{\veczeta(0,\cdot)}_{\wtd D}
=
\inpro{\sig\nabla\times\vecP}{\nabla\times\veczeta}_{\wtd D_T}
-
\inpro{\sig\nabla \times (e^{W(t)G}\vecm)}{\nabla\times\veczeta}_{D_T},
\end{equation}
where
\begin{equation}\label{equ:R}
 F(t,\vecm)
 =
 \lambda_1\vecm\times \wtd C(W(t),\vecm)
 -
 \lambda_2\vecm\times(\vecm\times \wtd C(W(t),\vecm))
 \end{equation}
for all $\vecxi\in L^2\big(0,T;\mW^{1,\infty}(D)\big)$ and
$\veczeta\in \CTE{\CX}$, with $\wtd C$ defined in Lemma~\ref{lem:eC wea}.
Then~$\vecM = e^{W(t)G}\vecm$ and 
$\vecH = \vecP - \wtd\vecM$ satisfy
~\eqref{wE:1.1}--\eqref{wE:Maxwell2} $\mP$-a.s.
\end{lemma}
\begin{proof}
\mbox{}

\noindent
\underline{{\it Step 1:} $\vecM$ and $\vecH$ satisfy~\eqref{wE:1.1}:}

Since $e^{W(t)G}$ is a semimartingale and $\vecm$ is absolutely continuous, 
using It\^o's formula for $\vecM=e^{W(t)G}\vecm$ (see e.g.~\cite{DaPrato92}),
we deduce
\begin{align}\label{equ:M Ito}
\vecM(t)
&=
\vecM(0)
+
\int_0^t Ge^{W(s)G}\vecm \dW(s)
+ 
\int_0^t \frac12 G^2 e^{W(s)G}\vecm \ds
+
\int_0^t e^{W(s)G} \vecm_t \ds
\nn
\\
&=
\vecM(0)
+
\int_0^t G\vecM \dW(s)
+ 
\dfrac12 \int_0^t G^2 \vecM \ds
+
\int_0^t e^{W(s)G} \vecm_t \ds,
\end{align}
where the first integral on the right-hand side is an It\^o integral and the
last two are Bochner integrals.
Recalling the relation between the Stratonovich and
It\^o differentials, namely
\begin{equation}\label{equ:Str Ito}
(G\vecu)\circ dW(s)
=
G\vecu \dW(s)
+
\dfrac12 G'(\vecu)[G\vecu] \ds,
\end{equation}
and noting that
$
G'(\vecu)[G\vecu] = G^2\vecu,
$
we rewrite~\eqref{equ:M Ito} in the Stratonovich form as
\[
\vecM(t)
=
\vecM(0)
+
\int_0^t G\vecM \circ dW(s) 
+
\int_0^t e^{W(s)G} \vecm_t\ds.
\]
Multiplying both sides of the above equation by a test function
$\vecpsi\in\C^{\infty}(D)$
and integrating over $D$ we obtain
\begin{align}\label{equ:wIto}
\inpro{\vecM(t)}{\vecpsi}_{D}
&=
\inpro{\vecM(0)}{\vecpsi}_{D}
+
\int_0^t
\inpro{G\vecM}{\vecpsi}_{D}\circ dW(s) 
+
\int_0^t
\inpro{e^{W(s)G}\vecm_t}{\vecpsi}_{D}\ds\nn\\
&=
\inpro{\vecM(0)}{\vecpsi}_{D}
+
\int_0^t
\inpro{G\vecM}{\vecpsi}_{D}\circ dW(s) 
+
\int_0^t
\inpro{\vecm_t}{e^{-W(s)G}\vecpsi}_{D}\ds,
\end{align}
where in the last step we used~\eqref{equ:eadj}.

On the other hand, we note that 
$e^{-W(\cdot)G}\vecpsi \in L^2\big(0,t;\mW^{1,\infty}(D)\big)$ for $t\in[0,T]$.
Let the test function $ \vecxi$  in~\eqref{InE:14} be $e^{-W(\cdot)G}\vecpsi $, 
we obtain from ~\eqref{equ:R} that
\begin{align*}
\int_0^t \inpro{\vecm_t}{e^{-W(s)G}\vecpsi}_{D} \ds
&=
-
\lambda_1
\int_0^t
\inpro{\vecm\times\nabla\vecm}{\nabla\left(e^{-W(s)G}\vecpsi\right)}_{D} \ds
\\
&\quad
-
\lambda_2
\int_0^t
\inpro{\vecm\times\nabla\vecm}%
{\nabla\left(\vecm\times\left(e^{-W(s)G}\vecpsi\right)\right)}_{D} \ds
\\
&\quad
+
\lambda_1
\int_0^t 
\inpro{\vecm\times\wtd C(W(s),\vecm)}{e^{-W(s)G}\vecpsi}_{D} \ds
\\
&\quad
-
\lambda_2
\int_0^t 
\inpro{\vecm\times\left(\vecm\times\wtd C(W(s),\vecm)\right)}%
{e^{-W(s)G}\vecpsi}_{D} \ds
\\
&\quad
+
\lambda_1 \int_0^t
\inpro{\vecm\times e^{-W(t)G}\vecP}{e^{-W(s)G}\vecpsi}_{D} \ds
\\
&\quad
-
\lambda_2 \int_0^t 
\inpro{\vecm\times(\vecm\times e^{-W(t)G}\vecP)}{e^{-W(s)G}\vecpsi}_{D} \ds
\\
&
=:
\int_0^t (T_1(s) + \cdots + T_6(s)) \ds.
\end{align*}
Considering $T_3$, we use successively~\eqref{equ:ele ide}, 
Lemma~\ref{lem:eC wea}, \eqref{equ:vecm}, and~\eqref{equ:G11} to obtain
\begin{align*}
T_3(s)
&
=
\lambda_1
\inpro{\vecm\times\wtd C(W(s),\vecm)}{e^{-W(s)G}\vecpsi}_D
=
-
\lambda_1
\inpro{\wtd C(W(s),\vecm)}{\vecm\times e^{-W(s)G}\vecpsi}_D
\\
&
=
-
\lambda_1
\inpro{\nabla\vecm}{\nabla\left(\vecm\times e^{-W(s)G}\vecpsi\right)}_D
+
\lambda_1
\inpro{\nabla\vecM}{\nabla\left(\vecM\times\vecpsi\right)}_D
\\
&
=
-
\lambda_1
\inpro{\nabla\vecm}{\vecm\times \nabla\left(e^{-W(s)G}\vecpsi\right)}_D
+
\lambda_1
\inpro{\nabla\vecM}{\vecM\times\nabla\vecpsi}_D
\\
&
=
\lambda_1
\inpro{\vecm\times \nabla\vecm}{\nabla\left(e^{-W(s)G}\vecpsi\right)}_D
-
\lambda_1
\inpro{\vecM\times\nabla\vecM}{\nabla\vecpsi}_D.
\end{align*}
Therefore,
\[
T_1 + T_3
=
-
\lambda_1
\inpro{\vecM\times\nabla\vecM}{\nabla\vecpsi}_D.
\]
Similarly, considering $T_4$ we have
\begin{align*}
T_4(s)
&
=
-
\lambda_2
\inpro{\vecm\times\left(\vecm\times\wtd C(W(s),\vecm)\right)}%
{e^{-W(s)G}\vecpsi}_{D} 
\\
&
=
\lambda_2
\inpro{\vecm\times\nabla\vecm}%
{\nabla\left(\vecm\times e^{-W(s)G}\vecpsi\right)}_D
-
\lambda_2
\inpro{\vecM\times\nabla\vecM}{\nabla(\vecM\times\vecpsi)}_D,
\end{align*}
so that
\[
T_2 + T_4
=
-
\lambda_2
\inpro{\vecM\times\nabla\vecM}{\nabla(\vecM\times\vecpsi)}_D.
\]
On the other hand, by using \eqref{equ:eadj},
\eqref{equ:G11}, and noting that $\vecP = \vecH + \vecM$ in $D$, we
obtain
\[
T_5(s)
=
\lambda_1
\inpro{\vecm\times e^{-W(s)G}\vecP}{e^{-W(s)G}\vecpsi}_D
=
\lambda_1
\inpro{\vecM\times\vecH}{\vecpsi}_D
\]
and
\[
T_6(s)
=
-
\lambda_2
\inpro{\vecm\times\left(\vecm\times e^{-W(s)G}\vecP\right)}{e^{-W(s)G}\vecpsi}_D
=
-
\lambda_2
\inpro{\vecM\times(\vecM\times\vecH)}{\vecpsi}_D.
\]
Therefore,
\begin{align*}
\int_0^t \inpro{\vecm_t}{e^{-W(s)G}\vecpsi}_{D} \ds
&=
-
\lambda_1
\int_0^t
\inpro{\vecM\times\nabla\vecM}{\nabla\vecpsi}_D \ds
-
\lambda_2
\int_0^t
\inpro{\vecM\times\nabla\vecM}{\nabla(\vecM\times\vecpsi)}_D \ds
\\
& \quad
+
\lambda_1
\int_0^t
\inpro{\vecM\times\vecH}{\vecpsi}_D \ds
-
\lambda_2
\int_0^t
\inpro{\vecM\times(\vecM\times\vecH)}{\vecpsi}_D \ds.
\end{align*}
This equation and~\eqref{equ:wIto} give
\begin{align*}
\inpro{\vecM(t)}{\vecpsi}_{D}
&=
\inpro{\vecM(0)}{\vecpsi}_{D}
+
\int_0^t
\inpro{G\vecM}{\vecpsi}_{D}\circ dW(s) 
\nn\\
&\quad
-
\lambda_1
\int_0^t
\inpro{\vecM\times\nabla\vecM}{\nabla\vecpsi}_D \ds
-
\lambda_2
\int_0^t
\inpro{\vecM\times\nabla\vecM}{\nabla(\vecM\times\vecpsi)}_D \ds
\\
& \quad
+
\lambda_1
\int_0^t
\inpro{\vecM\times\vecH}{\vecpsi}_D \ds
-
\lambda_2
\int_0^t
\inpro{\vecM\times(\vecM\times\vecH)}{\vecpsi}_D \ds.
\end{align*}
Hence, $\vecM$ and $\vecH$ satisfy~\eqref{wE:1.1}.

\medskip
\noindent
\underline{{\it Step 2:} $\vecM$ and $\vecH$ satisfy~\eqref{wE:Maxwell2}:}

This follows immediately from~\eqref{wE:vecP} and the fact that
\[
\inpro{\sig\nabla\times\wtd\vecM}{\nabla\times\veczeta}_{\wtd D_T}
=
\inpro{\sig\nabla\times\vecM}{\nabla\times\veczeta}_{D_T},
\]
completing the proof of the lemma.
\end{proof}
In the next lemma we provide an equivalence of equation~\eqref{InE:14},
namely its Gilbert form.
\begin{lemma}\label{lem:4.1}
Assume that $\vecm\in H^1\big(0,T;\mH^1(D)\big)$ and $\vecP\in\mL^2(\wtd D_T)$, 
$\mP$-a.s., satisfy
\begin{equation}\label{equ:m 1}
|\vecm(t,\cdot)| = 1, \quad t\in(0,T), \ a.e. \text{ in } D,\, \mP\text{-a.s.}
\end{equation}
Assume further that
$(\vecm,\vecP)$ satisfies $\mP$-a.s.
\begin{multline}\label{InE:13}
\lambda_1\inpro{\vecm\times\vecm_t}{\vecm\times\vecvarphi}_{D_T}
-
\lambda_2\inpro{\vecm_t}{\vecm\times\vecvarphi}_{D_T} 
-
\mu \inpro{\nabla\vecm}{\vecm\times\nabla\vecvarphi}_{D_T}
\\
-
\inpro{R(t,\vecm)}{\vecm\times\vecvarphi}_{D_T}
+
\mu\inpro{e^{-W(t)G}\vecP}{\vecm\times\vecvarphi}_{D_T}
=
0,
\end{multline}
for all $\vecvarphi\in L^2\big(0,T;\mH^1(D)\big)$, 
where 
$\mu=\lambda_1^2+\lambda_2^2$ and 
\begin{equation*}
 R(t,\vecm)
 =
 -
 \lambda_1^2\wtd C\big(W(t),\vecm\big)
+
 \lambda_2^2\vecm\times\big(\vecm\times \wtd C(W(t),\vecm)\big),
\end{equation*}
with $\wtd C$ defined in Lemma~\ref{lem:eC wea}.
Then $(\vecm,\vecP)$ satisfies~\eqref{InE:14} $\mP$-a.s.
\end{lemma}
\begin{proof}
Firstly, we observe that for each 
$\vecxi\in L^2\big(0,T;\mW^{1,\infty}(D)\big)$, due to 
Lemma~\ref{lem:4.0}, there exists 
$\vecvarphi\in L^2\big(0,T;\mH^1(D)\big)$ satisfying
\begin{equation}\label{equ:phi}
\vecxi
=
\lambda_1{\vecvarphi}
+
\lambda_2{\vecvarphi}\times\vecm.
\end{equation}

Next we derive some identities which will be used later in the proof. 
By using~\eqref{equ:abc} and noting~\eqref{equ:m 1} (so that
$\vecm\cdot\vecm_t=0$), we have 
\begin{equation}\label{equ:m mt m}
\vecm \times (\vecm \times \vecm_t) = - \vecm_t.
\end{equation}
Moreover,
\[
\vecm\times(\vecvarphi\times\vecm) 
= 
\vecvarphi - (\vecm\cdot\vecvarphi) \vecm
\quad\text{and}\quad 
\nabla\big(\vecm\times(\vecvarphi\times\vecm)\big)
= 
\nabla\vecvarphi 
- 
\nabla\big( (\vecm\cdot\vecvarphi) \vecm \big).
\]
The above identities and~\eqref{equ:nab nab} imply
\begin{equation}\label{equ:m P}
\big(
\vecm\times e^{-W(t)G}\vecP
\big)
\cdot
\big(
\vecm\times(\vecvarphi\times\vecm)
\big)
=
\big(
\vecm\times e^{-W(t)G}\vecP
\big)
\cdot
\vecvarphi
\end{equation}
and
\begin{align}\label{equ:m Del m}
(\vecm \times \nabla\vecm)
\cdot
\nabla\big(\vecm\times(\vecvarphi\times\vecm)\big)
&=
(\vecm \times \nabla\vecm)
\cdot
\nabla\vecvarphi
\nn
\\
&\quad
-
\sum_{i=1}^3
\left(
\vecm \times \frac{\pa\vecm}{\pa\ x_i}
\right)
\cdot
\left(
\frac{\pa(\vecm\cdot\vecvarphi)}{\pa x_i} \vecm
+
(\vecm\cdot\vecvarphi) \frac{\pa\vecm}{\pa x_i}
\right)
\nn
\\
&=
(\vecm \times \nabla\vecm)
\cdot
\nabla\vecvarphi,
\end{align}
where in the last step we used the elementary property
$(\veca\times\vecb)\cdot\veca = 0$ for all $\veca, \vecb\in\R^3$.

Now consider each term on the left-hand side of~\eqref{InE:14}.
By using~\eqref{equ:phi}--\eqref{equ:m Del m} and
noting~\eqref{equ:ele ide} we obtain
\begin{align*}
\inpro{\vecm_t}{\vecxi}_{D_T}
&=
\lambda_1
\inpro{\vecm_t}{\vecvarphi}_{D_T}
+
\lambda_2
\inpro{\vecm_t}{\vecvarphi\times\vecm}_{D_T}
\\
&=
-
\lambda_1
\inpro{\vecm\times(\vecm\times\vecm_t)}{\vecvarphi}_{D_T}
-
\lambda_2
\inpro{\vecm_t}{\vecm\times\vecvarphi}_{D_T}
\\
&=
\lambda_1
\inpro{\vecm\times\vecm_t}{\vecm\times\vecvarphi}_{D_T}
-
\lambda_2
\inpro{\vecm_t}{\vecm\times\vecvarphi}_{D_T},
\\
\lambda_1
\inpro{\vecm\times\nabla\vecm}{\nabla\vecxi}_{D_T}
&=
\lambda_1^2
\inpro{\vecm\times\nabla\vecm}{\nabla\vecvarphi}_{D_T}
+
\lambda_1\lambda_2
\inpro{\vecm\times\nabla\vecm}{\nabla(\vecvarphi\times\vecm)}_{D_T}
\\
&=
-
\lambda_1^2
\inpro{\nabla\vecm}{\vecm\times\nabla\vecvarphi}_{D_T}
+
\lambda_1\lambda_2
\inpro{\vecm\times\nabla\vecm}{\nabla(\vecvarphi\times\vecm)}_{D_T},
\\
\lambda_2
\inpro{\vecm\times\nabla\vecm}{\nabla(\vecm\times\vecxi)}_{D_T}
&=
\lambda_1\lambda_2
\inpro{\vecm\times\nabla\vecm}{\nabla(\vecm\times\vecvarphi)}_{D_T}
\\
&\quad
+
\lambda_2^2
\inpro{\vecm\times\nabla\vecm}%
{\nabla\big(\vecm\times(\vecvarphi\times\vecm)\big)}_{D_T}
\\
&=
-
\lambda_1\lambda_2
\inpro{\vecm\times\nabla\vecm}{\nabla(\vecvarphi\times\vecm)}_{D_T}
+
\lambda_2^2
\inpro{\vecm\times\nabla\vecm}{\nabla\vecvarphi}_{D_T}
\\
&=
-
\lambda_1\lambda_2
\inpro{\vecm\times\nabla\vecm}{\nabla(\vecvarphi\times\vecm)}_{D_T}
-
\lambda_2^2
\inpro{\nabla\vecm}{\vecm\times\nabla\vecvarphi}_{D_T},
\\
-
\inpro{F(t,\vecm)}{\vecxi}_{D_T}
&=
-
\lambda_1
\inpro{F(t,\vecm)}{\vecvarphi}_{D_T}
-
\lambda_2
\inpro{F(t,\vecm)}{\vecvarphi\times\vecm}_{D_T}
\\
&=
-
\lambda_1^2
\inpro{\vecm\times\wtd C(W(t),\vecm)}{\vecvarphi}_{D_T}
\\
&\quad
+
\lambda_2^2
\inpro{\vecm\times\big(\vecm\times\wtd C(W(t),\vecm)\big)}%
{\vecvarphi\times\vecm}_{D_T}
\\
&=
-
\inpro{R(t,\vecm)}{\vecm\times\vecvarphi}_{D_T},
\\
-
\lambda_1
\inpro{\vecm\times e^{-W(t)G}\vecP}{\vecxi}_{D_T}
&=
\lambda_1^2
\inpro{e^{-W(t)G}\vecP}{\vecm\times \vecvarphi}_{D_T}
\\
&\quad
-
\lambda_1\lambda_2
\inpro{\vecm\times e^{-W(t)G}\vecP}{\vecvarphi\times\vecm}_{D_T},
\\
\lambda_2
\inpro{\vecm\times\big(\vecm\times e^{-W(t)G}\vecP\big)}{\vecxi}_{D_T}
&=
\lambda_1\lambda_2
\inpro{\vecm\times e^{-W(t)G}\vecP}{\vecvarphi\times\vecm}_{D_T}
\\
&\quad
-
\lambda_2^2
\inpro{\vecm\times e^{-W(t)G}\vecP}%
{\vecm\times(\vecvarphi\times\vecm)}_{D_T}
\\
&=
\lambda_1\lambda_2
\inpro{\vecm\times e^{-W(t)G}\vecP}{\vecvarphi\times\vecm}_{D_T}
\\
&\quad
-
\lambda_2^2
\inpro{\vecm\times e^{-W(t)G}\vecP}{\vecvarphi}_{D_T}
\\
&=
\lambda_1\lambda_2
\inpro{\vecm\times e^{-W(t)G}\vecP}{\vecvarphi\times\vecm}_{D_T}
\\
&\quad
+
\lambda_2^2
\inpro{e^{-W(t)G}\vecP}{\vecm\times \vecvarphi}_{D_T}.
\end{align*}
Adding the above equations side by side we deduce that the left-hand side
of~\eqref{InE:14} equals that of~\eqref{InE:13}. 
Thus~\eqref{InE:14} holds if~\eqref{InE:13} holds. The lemma is proved.
\end{proof}
Thanks to Lemma~\ref{lem:we mP} and Lemma~\ref{lem:4.1}, in order to
solve~\eqref{E:1.1}--\eqref{E:Maxwell2}, we solve~\eqref{InE:13}
and~\eqref{wE:vecP}. It is therefore necessary to define the weak martingale
solutions for these two latter equations.
\begin{definition}\label{def:wea solm}
Given $T\in(0,\infty)$, a weak martingale solution
to~\eqref{InE:13} and~\eqref{wE:vecP} on the time interval $[0,T]$,
denoted by
$(\Omega,\cF,(\cF_t)_{t\in[0,T]},\mP,W,\vecm,\vecP)$,
consists of 
\begin{enumerate}
\renewcommand{\labelenumi}{(\alph{enumi})}
\item
a filtered probability space
$(\Omega,\cF,(\cF_t)_{t\in[0,T]},\mP)$ with the
filtration satisfying the usual conditions,
\item
a one-dimensional $(\cF_t)$-adapted Wiener process
$W=(W_t)_{t\in[0,T]}$,
\item 
a progressively measurable
process $\vecm : [0,T]\times\Omega \goto \mL^2(D)$,
\item
a progressively measurable
process $\vecP : [0,T]\times\Omega \goto \mL^2(\wtd D)$,
\end{enumerate}
such that there hold
\begin{enumerate}
\item
\quad
$\vecm \in \mH^1(D_T)$, $\mP$-a.s.;
\item
\quad
$\vecP \in L^2(0,T;\mH(\curl;\wtd D))$, $\mP$-a.s.;
\item
\quad
$\mE\left( 
\esssup_{t\in[0,T]}\|\nabla\vecm(t)\|^2_D
\right) < \infty$;
\item
\quad
$|\vecm(t,\cdot)| = 1$
for all $t \in [0,T]$, 
a.e. in $D$, 
and $\mP$-a.s.;
\item
\quad
$(\vecm,\vecP)$ satisfies~\eqref{InE:13}
and~\eqref{wE:vecP} $\mP$-a.s.
\end{enumerate}
\end{definition}
We state the following lemma which is a direct consequence of
Lemma~\ref{lem:we mP}, Lemma~\ref{lem:4.1},
and statement~\eqref{equ:m M 1}.
\begin{lemma}\label{lem:equi}
If $(\vecm,\vecP)$ is a weak martingale solution of~\eqref{InE:13} 
and~\eqref{wE:vecP} in the
sense of Definition~\ref{def:wea solm}, then $(\vecM,\vecH)$ is a weak
martingale solution of~\eqref{E:1.1} and~\eqref{E:Maxwell2} in the sense of
Definition~\ref{def:wea sol}.
\end{lemma}

In the next section, we
present a finite element scheme to approximate the solutions of~\eqref{InE:13}
and~\eqref{wE:vecP}.
\section{The finite element scheme}\label{sec:fin ele}
 In this section we introduce the $\theta$-linear finite element scheme 
 which approximates a weak solution $(\vecm,\vecP)$ defined in Definition
 ~\ref{def:wea solm}.
 
Let $\mT_h$ be a regular tetrahedrization of the domain~$\wtd D$ into
tetrahedra of maximal mesh-size~$h$. Let~$\mT_h|_D$ be  its
restriction to~$D\subset\wtd D$. We denote by~$\cN_h :=
\{\vecx_1,\ldots,\vecx_N\}$ the set of vertices in~$\mT_h|_D$ and
by~$\cM_h :=\{ \vece_1, \ldots , \vece_M \}$ the set of edges in~$\mT_h$.
 
 To discretize the equation~\eqref{InE:13}, we
 introduce the finite element space
 $\mV_h\subset\mH^1(D)$ defined by
\[
\mV_h:=
\left\{\vecu\in\mH^1(D) : 
\vecu|_K\in\big(P_1|_K\big)^3 \quad\forall K\in\mT_K\right\},
\]
where~$P_1$ is the set of polynomials of maximum total degree~$1$ 
in~$x_1,x_2,x_3$.
A basis for~$\mV_h$ can be chosen to 
be~$\{\phi_n\vecxi_1,\phi_n\vecxi_2,\phi_n\vecxi_3\}_{1\leq n\leq N}$, 
where~$\phi_n$ is a continuous piecewise linear function on~$\mT_h$ 
satisfying~$\phi_n(\vecx_m)=\delta_{n,m}$ (the Kronecker delta)
and~$\{\vecxi_j\}_{j=1,\cdots,3}$ is the canonical basis for~$\R^3$. 
The interpolation operator from
$\C^0(D)$ onto  $\mV_h$ is defined by
\[
I_{\mV_h}(\vecv)=\sum_{n=1}^N \vecv(\vecx_n)\phi_n(\vecx)
\quad\forall \vecv\in \mathbb C^0(D,\mathbb R^3) .
\]
 
To discretize~\eqref{wE:vecP}, we 
introduce the lowest order edge elements of N\'{e}d\'{e}lec's first
family (see~\cite{Monk03}) defined by
\[
\mY_h:=\left\{
\vecu\in\mH(\curl;\wtd D) : \vecu|_{K}\in\cD_K \quad \forall K\in\mT_h
\right\},
\]
where 
\[
\cD_K
:=
\left\{
\vecv:K\goto\R^3 \, : \,
\exists \veca,\vecb\in\R^3 \text{ such that }
\vecv(\vecx)=\veca+\vecb\times\vecx \quad
\forall \vecx\in K
\right\}.
\]
A basis $\{\vecpsi_1,\ldots,\vecpsi_M\}$ of $\mY_h$ can be defined by
\begin{equation*}\label{equ:psiq}
\int_{\vece_p}\vecpsi_q\cdot\vectau_p\ds 
= 
\delta_{q,p}\, ,
\end	{equation*}
where $ \vectau_p $ is the unit vector in the direction of edge
$\vece_p$. For any $\delta>0$ and $p>2$, the interpolation operator $I_{\mY_h}$
from $\mH^{1/2+\delta}(\wtd D)\cap\mW^{1,p}(\wtd D)$ onto $\mY_h$
is defined by
\[
I_{\mY_h}(\vecu)=\sum_{q=1}^M u_q \vecpsi_q
\quad\forall \vecu\in  \mH^{1/2+\delta}(\wtd D)\cap\mW^{1,p}(\wtd D),
\]
where
$
u_q=\int_{\vece_q}\vecu\cdot\vectau_q\ds.
$

Before introducing our approximation scheme,
we state the following result, proved in~\cite{Bart05}, 
which will be used in the analysis.
\begin{lemma}\label{lem:bar}
If there holds
\begin{equation}\label{E:CondTe}
\int_D \nabla\phi_i\cdot\nabla\phi_j\dvx \leq 0
\quad\text{for all}\quad i,j \in \{1,2,\cdots,N\}\text{ and
} i\not= j ,
\end{equation}
then for all $\vecu\in\mV_h$ satisfying
$|\vecu(\vecx_l)|\geq 1$, $ l=1,2,\cdots,N$, there holds
\begin{equation}\label{E:InE}
\int_D\left|\nabla
I_{\mV_h}\left(\frac{\vecu}{|\vecu|}\right)\right|^2\dvx
\leq
\int_D|\nabla\vecu|^2\dvx.
\end{equation}
\end{lemma}

When $d=2$, condition~\eqref{E:CondTe} holds for Delaunay
triangulations. When $d=3$, it holds if all dihedral angles of
the tetrahedra in $\mT_h|_D$ are less than or equal
to $\pi/2$; see~\cite{Bart05}.
In the sequel we assume that~\eqref{E:CondTe} holds.

With the finite element spaces defined as above, we are
ready to define our approximation scheme.
Fixing a positive integer $J$, we choose the time step
$k$ to be $k=T/J$ and define $t_j=jk$, $j=0,\cdots,J$.  For $j=1,2,\ldots,J$,
the functions $\vecm (t_j,\cdot)$ and $\vecP(t_j,\cdot)$ are 
approximated by $\vecm^{(j)}_h\in\mV_h$ and $\vecP_h^{(j)}\in\mY_h$, 
respectively. 
If~$\vecv_h^{(j)}$ is an approximation of~$\vecm_t(t_j,\cdot)$, then
since
\[
\vecm_t(t_j,\cdot)
\approx
\frac{\vecm(t_{j+1},\cdot)-\vecm(t_j,\cdot)}{k}
\approx
\frac{\vecm_h^{(j+1)}-\vecm_h^{(j)}}{k},
\]
we can define $\vecm_h^{(j+1)}$ from $\vecm_h^{(j)}$ by
\begin{equation}\label{equ:mjp1}
\vecm_h^{(j+1)}
:=
\vecm_h^{(j)} + k \vecv_h^{(j)},
\end{equation}
To maintain the condition $|\vecm_h^{(j+1)}|=1$, we
normalise
the right-hand side of~\eqref{equ:mjp1} and therefore
define $\vecm_h^{(j+1)}$ belonging to $\mV_h$ by
\[
\vecm_h^{(j+1)}
=
I_{\mV_h} \left(
\frac{\vecm_h^{(j)} + k \vecv_h^{(j)}}
{|\vecm_h^{(j)} + k \vecv_h^{(j)}|}
\right)
=
\sum_{n=1}^N
\frac{\vecm_h^{(j)}(\vecx_n)+k\vecv_h^{(j)}(\vecx_n)}
{\left|\vecm_h^{(j)}(\vecx_n)+k\vecv_h^{(j)}(\vecx_n)\right|}
\phi_n,
\]
which ensures that $|\vecm_h^{(j+1)}|=1$ at vertices.
Hence it suffices to propose a scheme to compute~$\vecv_h^{(j)}$.

We first rewrite~\eqref{InE:13} as
\begin{align}\label{equ:mt w}
\lambda_2\inpro{\vecm_t}{\vecw}_{D_T} 
&-
\lambda_1\inpro{\vecm\times\vecm_t}{\vecw}_{D_T}
+
\mu \inpro{\nabla\vecm}{\nabla\vecw}_{D_T}
\nn\\
&=
-
\inpro{R(t,\vecm)}{\vecw}_{D_T}
+
\mu\inpro{e^{-W(t)G}\vecP}{\vecw}_{D_T}
\end{align}
where $\vecw=\vecm\times\vecvarphi$.
Then, noting that $\vecm_t\cdot\vecm = 0$ (which follows from
$\snorm{\vecm}{}=1$) and $\vecw\cdot\vecm = 0$, we can design a
Galerkin method in which the unknown $\vecv_h^{(j)}$ and the test
function $\vecw_h$ reflect the above property.
Hence we follow~\cite{Alo08, AloJai06} to define
\begin{equation*}
 \mW_h^{(j)}
 :=
 \left\{\vecw\in \mV_h \mid 
 \vecw(\vecx_n)\cdot\vecm_h^{(j)}(\vecx_n)=0,
 \ n = 1,\ldots,N \right\},
\end{equation*}
and we will seek~$\vecv_h^{(j)}$ in this space. It remains to
approximate the other terms in~\eqref{equ:mt w}.

Considering the piecewise constant approximation $W_k(t)$
of $W(t)$, namely,
\begin{equation}\label{Def:W}
W_k(t)=W(t_j),\quad t\in[t_j,t_{j+1}),
\end{equation}
we define
\begin{align}
\vecg_h
&:=
I_{\mV_h}(\vecg),
\nn\\
G_h\vecu 
&:= 
\vecu\times\vecg_h
\quad \forall\vecu\in\mV_h\cup\mY_h,
\nn\\
e^{ W_k(t)G_h}\vecu
&:= 
\vecu
+
(\sin W_k(t))G_h\vecu
+
(1-\cos W_k(t))G_h^2\vecu
\quad \forall\vecu\in\mV_h\cup\mY_h,\label{def:eh}
\\
C_h(\vecu)
&:=
\vecu\times I_{\mV_h}(\Delta\vecg)
+
2\nabla\vecu\times I_{\mV_h}(\nabla\vecg)
\quad \forall\vecu\in\mV_h,
\nn
\\
D_{h,k}(t,\vecu)
&=
\Big(
\big(\sin W_k(t)\big) C_h + \big(1-\cos W_k(t)\big)(G_hC_h+C_hG_h) 
\Big)\vecu \label{equ:Dhk} 
\\
\wtd C_{h,k}(t,\vecu)
&=
\Big(I-\sin W_k(t) G_h + (1-\cos
W_k(t))G_h^2\Big)D_{h,k}(t,\vecu),  \label{equ:Chk}
\\
R_{h,k}(t,\vecu)
&=
\lambda_2^2 \vecu\times(\vecu\times
\wtd C_{h,k}(t,\vecu))
-
\lambda_1^2 \wtd C_{h,k}(t,\vecu). \label{equ:Fk}
\end{align}
We can now discretise~\eqref{equ:mt w} as: For some $\theta\in[0,1]$,
find $\vecv_h^{(j)}\in\mW_h^{(j)}$ satisfying
\begin{align}\label{disE:LLG}
\lambda_2
\inpro{\vecv_h^{(j)}}{\vecw_h^{(j)}}_{D}
&
-
\lambda_1
\inpro{\vecm_h^{(j)}\times\vecv_h^{(j)}}{\vecw_h^{(j)}}_{D}
+
\mu 
\inpro{\nabla (\vecm_h^{(j)}+k\theta \vecv_h^{(j)})}%
{\nabla\vecw_h^{(j)}}_{D} 
\nn
\\
&=
-
\inpro{R_{h,k}(t_j,\vecm_h^j)}{\vecw_h}_{D}
+
\mu
\inpro{e^{-W_k(t_j)G_h}\vecP^{(j)}_h}{\vecw_h^{(j)}}_{D}
\quad\forall \vecw_h^{(j)}\in\mW_h^{(j)}.
\end{align}

To discretise~\eqref{wE:vecP}, even though $\vecP$ is not time
differentiable we formally use integration by parts to
bring the time derivative to~$\vecP$, and thus with $d_t\vecP^{(j+1)}_h$
defined by
\[
d_t\vecP^{(j+1)}_h
:=k^{-1}\big(\vecP^{(j+1)}_h-\vecP^{(j)}_h\big),
\]
the discretisation of~\eqref{wE:vecP} reads:
Compute~$\vecP_h^{(j+1)}\in\mY_h$ by solving
\begin{align}\label{disE:NewMax2}
 \mu_0\inpro{d_t\vecP^{(j+1)}_h}{\veczeta_h}_{\wtd D}
+
 \inpro{\sigma\nabla\times\vecP^{(j+1)}_h}
 {\nabla\times\veczeta_h}_{\wtd D}
=
 \sigma_D\inpro{\nabla\times\big(e^{W_k(t_{j})G_h}\vecm^{(j)}_h\big)}
 {\nabla\times\veczeta_h}_{D}
 \quad\forall \veczeta_h\in\mY_h.
\end{align}

We summarise the above procedure in the following algorithm.
\begin{algorithm}\label{Algo:1}
\begin{description}
\mbox{}
\item[Step 1]
Set $j=0$.
Choose $\vecm^{(0)}_h=I_{\mV_h}\vecm_0$ and
$\vecP^{(0)}_h=I_{\mY_h}\vecP_0$.
\item[Step 2] \label{A:2}
Solve~\eqref{disE:LLG} and~\eqref{disE:NewMax2} to
find $(\vecv_h^{(j)},\vecP^{(j+1)}_h)\in \mW_h^{(j)}\times\mY_h$.
\item[Step 3] \label{A:4}
Define
\begin{equation*}
\vecm_h^{(j+1)}(\vecx)
:=
\sum_{n=1}^N
\frac{\vecm_h^{(j)}(\vecx_n)+k\vecv_h^{(j)}(\vecx_n)}
{\left|\vecm_h^{(j)}(\vecx_n)+k\vecv_h^{(j)}(\vecx_n)\right|}
\phi_n(\vecx).
\end{equation*}
\item[Step 4] 
Set $j=j+1$ and 
and return to Step $2$ if $j<J$. Stop if $j=J$.
\end{description}
\end{algorithm}
By the Lax--Milgram theorem, for each $j>0$ 
there exists a unique solution 
$(\vecv_h^{(j)},\vecP^{(j+1)}_h)\in \mW_h^{(j)}\times\mY_h$ 
of equations~\eqref{disE:LLG}--\eqref{disE:NewMax2}.
Since $\left|\vecm_h^{(0)}(\vecx_n)\right|=1$ and 
$\vecv_h^{(j)}(\vecx_n)\cdot\vecm_h^{(j)}(\vecx_n)=0$ for all
$n=1,\ldots,N$ and $j=0,\ldots,J$, there hold (by induction)
\begin{equation}\label{equ:mhj 1}
\left |\vecm_h^{(j)}(\vecx_n)+k\vecv_h^{(j+1)}(\vecx_n)\right| \ge 1
\quad\text{and}\quad
\left |\vecm_h^{(j)}(\vecx_n)\right |=1,
\quad j = 0,\ldots,J.
\end{equation}
In particular, the above inequality shows that Step 3 of
the algorithm is well defined.

We finish this section by proving the following lemmas
concerning boundedness of
$\vecm_h^{(j)}$, $\vecP_h^{(j)}$ and $R_{h,k}$.
\begin{lemma}\label{lem:mhj}
For any $j=0,\ldots,J$ there hold
\[
\norm{\vecm_h^{(j)}}{\mL^{\infty}(D)} \le 1
\quad\text{and}\quad 
\norm{\vecm_h^{(j)}}{D} \le |D|,
\]
where $|D|$ denotes the measure of $D$.
\end{lemma}
\begin{proof}
The first inequality follows from~\eqref{equ:mhj 1} and the
second can be obtained by integrating over $D$.
\end{proof}
\begin{lemma}\label{lem:Fk}
Assume that $\vecg$ satisfies~\eqref{equ:g 1} and $\vecg\in\mW^{2,\infty}(D)$.
There exists a deterministic constant $c$ depending only on $\vecg$
such that, for any $j=0,\cdots,J$, there holds $\mP\text{-a.s.}$,
\begin{align}
\left\|R_{h,k}(t_j,\vecm_h^{(j)})\right\|_D^2
&\leq
c + 
c\left \| \nabla \vecm_h^{(j)} \right \| _D^2,\label{equ:boundFk}\\
\|e^{-W_k(t_j)G_h}\vecu\|^2_D
&\leq
\|\vecu\|^2_D\quad\forall\vecu\in\mL^2(D),\label{equ:boundP}\\
\|\nabla\times\big(e^{W_k(t_j)G_h}\vecm^{(j)}_h\big)\|_D^2
&\leq
c
+
c\|\nabla\vecm^{(j)}_h\|_D^2.\label{equ:boundcurlP}
\end{align}
\end{lemma}
\begin{proof}
The proof of~\eqref{equ:boundFk} is similar to that of~\cite[Lemma 5.3]{BNT13}. 
To prove~\eqref{equ:boundP} we first note that the definition of 
$e^{-W_k(t_j) G_h}\vecu$ gives
\begin{align*}
\left|e^{-W_k(t_j)G_h}\vecu\right|^2
&=
\left|
\vecu
-
\big(\sin W_k(t_j)\big) \vecu\times\vecg_h
+
\big(1-\cos W_k(t_j)\big) (\vecu\times\vecg_h)\times\vecg_h
\right|^2 
\\
&=
|\vecu|^2
+
(1-\cos W_k(t_j) )^2
\left(\left|\left(\vecu\times\vecg_h\right)\times\vecg_h\right|^2
-
\left|\vecu\times\vecg_h\right|^2
\right)\\
&=
|\vecu|^2
+
\big(1-\cos W_k(t_j)\big)^2
(|\vecg_h|^2-1)
\left|\vecu\times\vecg_h\right|^2,
\end{align*}
where in the last step we used $|(\veca\times\vecb)\times\vecb|^2 =
|\veca\times\vecb|^2 |\vecb|^2$ for all $\veca$, $\vecb\in\R^3$.
Since $|\vecg(\vecx_i)|=1$
and $\sum_{i=1}^N \phi_{i}(\vecx)=1$ for all $\vecx\in D$, we have
\[
|\vecg_h(\vecx)|^2
=
\left|\sum_{i=1}^{N} 
\vecg(\vecx_{i})\phi_{i}(\vecx)\right|^2
\le 1.
\]
Therefore,
\begin{equation*}
\left|e^{-W_k(t_j) G_h}\vecu\right|^2
\leq
\left|\vecu\right|^2
\quad a.e. \text{ in } D,
\end{equation*}
proving~\eqref{equ:boundP}.

Finally, in order to prove~\eqref{equ:boundcurlP} we use
the inequality
\[
\norm{\nabla\times\vecu}{D}^2
\leq
c\norm{\nabla\vecu}{D}^2
\quad\forall\vecu\in\mH^1(D)
\]
to obtain
\begin{align*}
\norm{\nabla\times (e^{W_k(t_j) G_h}\vecm^{(j)}_h)}{D}^2
\leq
c\norm{\nabla\big(e^{W_k(t_j) G_h}\vecm^{(j)}_h\big)}{D}^2.
\end{align*}
On the other hand from the definition of $e^{W_k(t_j) G_h}$ 
, we deduce
\begin{align*}
\nabla\big(e^{W_k(t_j) G_h}\vecm^{(j)}_h\big)
&=
e^{W_k(t_j) G_h}\nabla\vecm^{(j)}_h
+
\big(\sin W_k(t_j)\big)\vecm^{(j)}_h\times\nabla\vecg_h\\
&\quad+
(1-\cos W_k(t_j) )
\left(
(\vecm^{(j)}_h\times\nabla\vecg_h)\times\vecg_h
+
(\vecm^{(j)}_h\times\vecg_h)\times\nabla\vecg_h
\right).
\end{align*}
Since $\vecg\in\mW^{2,\infty}(D)$, by using Lemma~\ref{lem:mhj} and~\eqref{equ:boundP} we obtain from the above equality
\begin{align*}
\left|
\nabla (e^{W_k(t_j) G_h}\vecm^{(j)}_h)
\right|^2
&\leq
c
+
\left|
e^{W_k(t_j) G_h}(\nabla\vecm^{(j)}_h)
\right|^2
\leq
c
+
c
\left|
\nabla\vecm^{(j)}_h
\right|^2.
\end{align*}
This completes the proof.
\end{proof}
\begin{lemma}\label{lem:4.2}
The sequence 
$\left\{
\big(\vecm_h^{(j)},\vecv_h^{(j)},\vecP^{(j)}_h\big)
\right\}_{j=0,1,\cdots,J}$ produced by Algorithm \ref{Algo:1} satisfies
$\mP\text{-a.s.},$
\begin{align}\label{InE:mP3}
\|\nabla\vecm^{(j)}_h\|^2_{D}
&+
k
\sum_{i=0}^{j-1}
\|\vecv_h^{(i)}\|^2_{D}
+k^2(2\theta-1)
\sum_{i=0}^{j-1}
\|\nabla\vecv^{(i)}_h\|^2_{D}
+
\|\vecP_h^{(j)}\|^2_{\wtd D}
\nn\\
&
+
\sum_{i=0}^{j-1}
\|\vecP_h^{(i+1)}-\vecP_h^{(i)}\|^2_{\wtd D}
+
k
\sum_{i=0}^{j-1}
\|\nabla\times\vecP^{(i)}_h\|^2_{\wtd D}
\leq
c.
\end{align}
\end{lemma}
\begin{proof}
Choosing $\vecw^{(j)}_h=\vecv^{(j)}_h$ in~\eqref{disE:LLG}, we obtain
\begin{align*}
\lambda_2
\|\vecv_h^{(j)}\|^2_{D}
+
\mu k\theta
\|\nabla\vecv_h^{(j)}\|^2_{D}
&=
-\mu\inpro{\nabla\vecm_h^{(j)}}{\nabla\vecv_h^{(j)}}_{D} 
-
\inpro{R_{h,k}(t_j,\vecm_h^j)}{\vecv_h^{(j)}}_D
\\
&\quad
+
\mu
\inpro{e^{-W_k(t_j) G_h}\vecP^{(j)}_h}{\vecv_h^{(j)}}_{D},
\end{align*}
or equivalently
\begin{align*}
\inpro{\nabla\vecm_h^{(j)}}{\nabla\vecv_h^{(j)}}_{D}
&=
-\lambda_2\mu^{-1}
\|\vecv_h^{(j)}\|^2_{D}
-k\theta
\|\nabla\vecv_h^{(j)}\|^2_{D}
-\mu^{-1}
\inpro{R_{h,k}(t_j,\vecm_h^j)}{\vecv_h^{(j)}}_D
\nn\\
&\quad
+
\inpro{e^{-W_k(t_j) G_h}\vecP^{(j)}_h}{\vecv_h^{(j)}}_{D}.
\end{align*}
Lemma~\ref{lem:bar} and the above equation yield
\begin{align*}
\|\nabla\vecm_h^{(j+1)}\|_D^2
&\leq
\|\nabla(\vecm_h^{(j)}+k\vecv_h^{(j)})\|_D^2
\\
&=
\|\nabla\vecm_h^{(j)}\|_D^2
+
k^2(1-2\theta)
\|\nabla\vecv_h^{(j)}\|_D^2
-
2k\lambda_2\mu^{-1}
\|\vecv_h^{(j)}\|^2_{D}
\\
&\quad
-2k\mu^{-1}
\inpro{R_{h,k}(t_j,\vecm_h^j)}{\vecv_h^{(j)}}_D
+
2k
\inpro{e^{-W_k(t_j) G_h}\vecP^{(j)}_h}{\vecv_h^{(j)}}_{D}.
\end{align*}
By using the elementary inequality
\begin{equation}\label{InE:Cauchy}
2ab\leq \alpha^{-1}a^2+\alpha b^2
\quad \forall\alpha>0, \forall a, b\in\R,
\end{equation} 
for the last two terms on the right hand side, we deduce
\begin{multline*}
\|\nabla \vecm_h^{(j+1)}\|_D^2
+
2k\lambda_2\mu^{-1}
\|\vecv_h^{(j)}\|^2_{D}
+k^2(2\theta-1)
\|\nabla\vecv_h^{(j)}\|^2_{D}
\\
\leq
\|\nabla\vecm_h^{(j)}\|_D^2
+
k\lambda_2\mu^{-1}
\|\vecv_h^{(j)}\|_D^2
+2k\lambda_2^{-1}\mu^{-1}
\|R_{h,k}(t_j,\vecm_h^j)\|_D^2
+
2k\lambda_2^{-1}\mu
\|e^{-W_k(t_j) G_h}\vecP^{(j)}_h\|_D^2.
\end{multline*}
By rearranging the above inequality and using~\eqref{equ:boundFk}--\eqref{equ:boundP} 
we obtain
\begin{align*}
\|\nabla\vecm_h^{(j+1)}\|_D^2
&+
k\lambda_2\mu^{-1}
\|\vecv_h^{(j)}\|^2_{D}
+k^2(2\theta-1)
\|\nabla\vecv_h^{(j)}\|^2_{D}\nn\\
&\leq
\|\nabla\vecm_h^{(j)}\|_D^2
+
2k\lambda_2^{-1}\mu
\|\vecP^{(j)}_h\|_D^2
+
k\lambda_2^{-1}\mu^{-1}c
\|\nabla\vecm_h^j\|_D^2
+
k\lambda_2^{-1}\mu^{-1}c.
\end{align*}
Replacing $j$ by $i$ in the above inequality and summing for $i$ from $0$ to
$j-1$  yields
\begin{align*}
\|\nabla\vecm_h^{(j)}\|_D^2
&+
\lambda_2\mu^{-1}
k
\sum_{i=0}^{j-1}
\|\vecv_h^{(i+1)}\|^2_{D}
+(2\theta-1)
k^2
\sum_{i=0}^{j-1}
\|\nabla\vecv_h^{(i+1)}\|^2_{D}\nn\\
&\leq
\|\nabla\vecm_h^{(0)}\|_D^2
+
ck
\sum_{i=0}^{j-1}
\|\vecP^{(i)}_h\|_D^2
+
ck
\sum_{i=0}^{j-1}
\|\nabla\vecm_h^i\|_D^2
+
c.
\end{align*}
Since $\vecm_0\in\mH^2(D)$ it can be shown that there exists a deterministic constant 
$c$ depending only on $\vecm_0$ such that 
\begin{equation}\label{equ:boundm0}
\|\nabla\vecm_h^{(0)}\|_D^2
\leq c.
\end{equation}
By using~\eqref{equ:boundm0} we deduce
\begin{align}\label{InE:mP1}
\|\nabla\vecm_h^{(j)}\|_D^2
&+
k
\sum_{i=0}^{j-1}
\|\vecv_h^{(i+1)}\|^2_{D}
+
k^2
(2\theta-1)
\sum_{i=0}^{j-1}
\|\nabla\vecv_h^{(i+1)}\|^2_{D}\nn\\
&\leq
c
+
c
\sum_{i=0}^{j-1}k
\|\vecP^{(i)}_h\|_D^2
+c
\sum_{i=0}^{j-1}k
\|\nabla\vecm_h^i\|_D^2.
\end{align}
In order to estimate the two sums on the right-hand side, we take
$\veczeta_h=\vecP^{(j+1)}_h$ in~\eqref{disE:NewMax2}
to obtain the following identity
\begin{align*}
\mu_0
\inpro{d_t \vecP^{(j+1)}_h}{\vecP^{(j+1)}_h}_{\wtd D}
+
\inpro{\sigma\nabla\times\vecP^{(j+1)}_h}{\nabla\times\vecP^{(j+1)}_h}_{\wtd D}
=
 \sigma_D\inpro{\nabla\times (e^{W_k(t_j) G_h}\vecm^{(j)}_h)}
 {\nabla\times\vecP^{(j+1)}_h}_{D}.
\end{align*}
Let $\sig_0$ is the lower bound of $\sig$ on $\wtd D$.
By using successively~\eqref{InE:Cauchy} and~\eqref{equ:boundcurlP} we deduce from the above equality
\begin{align*}
\mu_0
\inpro{\vecP_h^{(j+1)}-\vecP_h^{(j)}}{\vecP_h^{(j+1)}}_{\wtd D}
+
k\sigma_0
\|\nabla\times\vecP^{(j+1)}_h\|^2_{\wtd D}
&\leq
+
k\frac{\sigma_D^2}{2\sigma_0}
\|\nabla\times (e^{W_k(t_j) G_h}\vecm^{(j)}_h)\|_D^2\nn\\
&\quad
+
\frac{1}{2}k\sigma_0
\|\nabla\times\vecP^{(j+1)}_h\|^2_{D}\nn\\
&\leq
\frac{1}{2}k\sigma_0
\|\nabla\times\vecP^{(j+1)}_h\|^2_{\wtd D}\nn\\
&\quad+
ck
\|\nabla\vecm^{(j)}_h\|_D^2
+ck,
\end{align*}
or equivalently
\begin{align*}
\mu_0
\inpro{\vecP_h^{(j+1)}-\vecP_h^{(j)}}{\vecP_h^{(j+1)}}_{\wtd D}
+
\frac{1}{2}k\sigma_0
\|\nabla\times\vecP^{(j+1)}_h\|^2_{\wtd D}
\leq
ck
\|\nabla\vecm^{(j)}_h\|_D^2
+ck.
\end{align*}
Replacing $j$ by $i$ in the above inequality and summing over $i$ from $0$ to
$j-1$ and using the following Abel summation
\begin{equation*}
\sum_{i=0}^{j-1}
(\veca_{i+1}-\veca_i)\cdot\veca_{i+1}
=
\frac{1}{2} |\veca_{j}|^2
-
\frac{1}{2} |\veca_{0}|^2
+
\frac{1}{2}\sum_{i=0}^{j-1} |\veca_{i+1}-\veca_i|^2,
\quad \veca_i \in \R^3,
\end{equation*}
we obtain
\begin{align*}
\|\vecP_h^{(j)}\|^2_{\wtd D}
+
\sum_{i=0}^{j-1}
\|\vecP_h^{(i+1)}-\vecP_h^{(i)}\|^2_{\wtd D}
&+
\sigma_0\mu_0^{-1}
\sum_{i=0}^{j-1}k
\|\nabla\times\vecP^{(i+1)}_h\|^2_{\wtd D}\nn\\
&\leq
\|\vecP_h^{(0)}\|^2_{\wtd D}
+
c\sum_{i=0}^{j-1}k
\|\nabla\vecm^{(i)}_h\|^2_{D}
+cT\sigma.
\end{align*}
By using~\eqref{E:Cond1} and the error estimate for the interpolant
$\vecP_h^{(0)} = I_{\mY_h}\vecP_0$, it can be shown that there exists a constant c depending only on
$\vecP_0$ such that
\begin{equation}\label{equ:boundP0}
\|\vecP_h^{(0)}\|^2_{\wtd D}
+
\|\nabla\times\vecP_h^{(0)}\|^2_{\wtd D}
\leq
c.
\end{equation}
By using~\eqref{equ:boundP0} we deduce
\begin{align}\label{InE:mP2}
\|\vecP_h^{(j)}\|^2_{\wtd D}
+
\sum_{i=0}^{j-1}
\|\vecP_h^{(i+1)}-\vecP_h^{(i)}\|^2_{\wtd D}
&+
k
\sum_{i=0}^{j-1}
\|\nabla\times\vecP^{(i+1)}_h\|^2_{\wtd D}
&\leq
c
+
ck
\sum_{i=0}^{j-1}
\|\nabla\vecm^{(i)}_h\|^2_{D}.
\end{align}

From~\eqref{InE:mP1} and~\eqref{InE:mP2} we obtain
\begin{align*}
\|\nabla\vecm^{(j)}_h\|^2_{D}
+
\|\vecP_h^{(j)}\|^2_{\wtd D}
\leq
c+
ck
\sum_{i=0}^{j-1}
\|\vecP_h^{(i)}\|^2_{\wtd D}
+
ck
\sum_{i=0}^{j-1}
\|\nabla\vecm^{(i)}_h\|^2_{\wtd D}.
\end{align*}
By using induction and~\eqref{equ:boundm0}-\eqref{equ:boundP0} we can show that
\begin{equation*}
\|\nabla\vecm^{(i)}_h\|^2_{D}
+
\|\vecP_h^{(i)}\|^2_{\wtd D}
\leq
c(1+ck)^i.
\end{equation*}
Summing over $i$ from $0$ to $j-1$ and using $1+x \leq e^x$ we obtain
\begin{equation}\label{InE:Gronwall}
k\sum_{i=0}^{j-1}
\|\nabla\vecm^{(i)}_h\|^2_{D}
+
k\sum_{i=0}^{j-1}
\|\vecP_h^{(i)}\|^2_{\wtd D}
\leq
ck\frac{(1+ck)^j-1}{ck}
\leq
e^{ckJ}=c.
\end{equation}
The required result~\eqref{InE:mP3} now follows from~\eqref{InE:mP1},~\eqref{InE:mP2} 
and~\eqref{InE:Gronwall}.
\end{proof}
\section{Proof of the main theorem}\label{sec:pro}
 The discrete solutions $\vecm_h^{(j)}$, $\vecv_h^{(j)}$ 
 and $\vecP_h^{(j)}$
 constructed via Algorithm~\ref{Algo:1}
 are interpolated in time in the following definition.
 \begin{definition}\label{def:mhk}
 For all $x\in D$ and all $t\in[0,T]$, let 
 $j\in \{ 0,...,J-1 \}$ be
 such that  $t \in [t_j, t_{j+1})$. We then define
 \begin{align*}
 \vecm_{h,k}(t,\vecx) 
 &:= 
 \frac{t-t_j}{k}\vecm_h^{(j+1)}(\vecx)
 +
 \frac{t_{j+1}-t}{k}\vecm_h^{(j)}(\vecx), \\
 \vecm_{h,k}^{-}(t,\vecx) 
 &:= 
 \vecm_h^{(j)}(\vecx), \\
 \vecv_{h,k}(t,\vecx) 
 &:= 
 \vecv_h^{(j)}(\vecx),\\
  \vecP_{h,k}(t,\vecx) 
 &:= 
 \frac{t-t_j}{k}\vecP_h^{(j+1)}(\vecx)
 +
 \frac{t_{j+1}-t}{k}\vecP_h^{(j)}(\vecx), \\
 \vecP_{h,k}^{-}(t,\vecx) 
 &:= 
 \vecP_h^{(j)}(\vecx),\\
 \vecP_{h,k}^{+}(t,\vecx) 
 &:= 
 \vecP_h^{(j+1)}(\vecx).
 \end{align*}
 \end{definition}
 The above sequences have the following obvious bounds.
 \begin{lemma}\label{lem:3.2a}
 There exist a deterministic constant $c$
 depending on  $\vecm_0$, $\vecP_0$, $\vecg$, $\mu$, $\sigma$ and $T$
 such that for all $\theta\in[0,1]$ there holds $\mP$-a.s.
 \begin{align*}
 \norm{\vecm_{h,k}^*}{D_T}^2 +
 \left \| \nabla \vecm_{h,k}^* \right \| _{D_T}^2
 +
 \left \| \vecv_{h,k}\right\|_{D_T}^2
 +
 k (2\theta-1)
 \left\| \nabla \vecv_{h,k} \right\|_{D_T}^2 
 \leq
 c,
 \end{align*}
 where $\vecm_{h,k}^*=\vecm_{h,k}$ or $\vecm_{h,k}^-$.
 In particular, 
 when $\theta\in[0,\frac{1}{2})$, there holds  $\mP$-a.s.
 \begin{align*}
 \norm{\vecm_{h,k}^*}{D_T}^2 
 +
 &\left \| \nabla \vecm_{h,k}^* \right \| _{D_T}^2
 +
 \big(1+(2\theta-1)kh^{-2}\big)
 \left \| \vecv_{h,k}\right\|_{D_T}^2
 \leq
 c.
 \end{align*}
 \end{lemma}
 \begin{proof}
 Both inequalities are direct consequences of Definition~\ref{def:mhk}, 
 Lemmas~\ref{lem:mhj} and~\ref{lem:4.2}, noting that
the second inequality requires the use of
the inverse estimate (see e.g.~\cite{Johnson87})
\[
\norm{\nabla\vecv_{h}^{(i)}}{D}^2
\leq
ch^{-2}
\norm{\vecv_{h}^{(i)}}{D}^2.
\]
 \end{proof}
  \begin{lemma}\label{lem:3.2b}
 There exist a deterministic constant $c$
 depending on $\vecm_0$, $\vecP_0$, $\vecg$, $\mu$, $\sigma$ and $T$
 such that for all $\theta\in[0,1]$ there holds $\mP$-a.s.
 \begin{align}
 \norm{\vecP_{h,k}}{\wtd D_T}^2 
 +
 \norm{\vecP_{h,k}^+}{\wtd D_T}^2
 +
 \norm{\nabla\times\vecP_{h,k}^+}{\wtd D_T}^2
&\leq
 c,\label{equ:Phk1}\\
\norm{\vecP_{h,k}-\vecP_{h,k}^*}{\wtd D_T}^2 
&\leq
kc,\label{equ:Phk2}
 \end{align}
 where $\vecP_{h,k}^*=\vecP_{h,k}^+$ or $\vecP_{h,k}^-$.
 \end{lemma}
\begin{proof}
It is easy to prove~\eqref{equ:Phk1} by using 
Lemma~\ref{lem:4.2} and Definition~\ref{def:mhk}.
Inequality~\eqref{equ:Phk2} can be deduced from Lemma~\ref{lem:4.2} by noting
that for $t\in[t_j,t_{j+1})$ there holds
\begin{align*}
\left|
 \vecP_{h,k}(t,\vecx)-\vecP_{h,k}^+(t,\vecx)
 \right|
 =
\left|
 \frac{t-t_{j+1}}{k}
 \big(
 \vecP_h^{(j+1)}(\vecx)-\vecP_h^{(j)}(\vecx)
 \big)
 \right|
 \leq
 \left|\vecP_h^{(j+1)}(\vecx)-\vecP_h^{(j)}(\vecx)\right|,
\end{align*}
completing the proof of the lemma.
\end{proof}
The next lemma provides a bound of $\vecm_{h,k}$ in the 
 $\mH^1$-norm and relationships between $\vecm_{h,k}^-$, 
 $\vecm_{h,k}$ and $\vecv_{h,k}$.
 \begin{lemma}\label{lem:3.4}
 Assume that $h$ and $k$ go to $0$ with a further condition $k=o(h^2)$ when
$\theta\in[0,\frac{1}{2})$ and no condition otherwise. The sequences
$\{\vecm_{h,k}\}$, $\{\vecm_{h,k}^{-}\}$, and
 $\{\vecv_{h,k}\}$ defined in
 Definition~\ref{def:mhk} satisfy the following properties
 $\mP$-a.s.
 \begin{align}
 \norm{\vecm_{h,k}}{\mH^1(D_T)} 
 &\le c, \label{equ:mhk h1} \\
 \norm{\vecm_{h,k}-\vecm_{h,k}^-}{D_T}
 &\le ck, \label{equ:mhk mhkm} \\
  \norm{\vecv_{h,k}-\pa_t\vecm_{h,k}}{\mL^1(D_T)}
 &\le ck, \label{equ:vhk mhk} \\
 \norm{|\vecm_{h,k}|-1}{D_T}
 &\le chk. \label{equ:mhk 1}
 \end{align}
 \end{lemma} 
\begin{proof}
The proof of this lemma is similar to that of~\cite[Lemma 6.3]{BNT13}
\end{proof}
 The following two Lemmas~\ref{lem:3.7a} show that~$\vecm_{h,k}$ and
~$\vecP_{h,k}$, respectively, satisfy discrete forms of~\eqref{InE:13}
and~\eqref{wE:vecP}.
 \begin{lemma}\label{lem:3.7a}
 Assume that $h$ and $k$ go to 0 with the following conditions
\begin{equation}\label{equ:theta}
\begin{cases}
k = o(h^2) & \quad\text{when } 0 \le \theta < 1/2, \\
k = o(h) & \quad\text{when } \theta = 1/2, \\
\text{no condition} & \quad\text{when } 1/2<\theta\le1.
\end{cases}
\end{equation}
Then for any $\vecvarphi \in C\big(0,T;\C^{\infty}(D)\big)$ and
$\veczeta \in C_T^{1}(0, T;\CX)$, there holds $\mP$-a.s.
 \begin{align}\label{InE:10}
 &-\lambda_1\inpro{\vecm_{h,k}^-\times\vecv_{h,k}}
 {\vecm_{h,k}^-\times\vecvarphi}_{D_T}
 +
 \lambda_2\inpro{\vecv_{h,k}}
 {\vecm_{h,k}^-\times\vecvarphi}_{D_T}\nn\\
&+
 \mu\inpro{\nabla(\vecm_{h,k}^-+k\theta\vecv_{h,k})}
 {\nabla(\vecm_{h,k}^-\times\vecvarphi)}_{D_T} 
+
 \inpro{R_{h,k}(\cdot,\vecm_{h,k}^-)}
 {\vecm_{h,k}^-\times\vecvarphi}_{D_T}\nn\\
&-\mu
\inpro{e^{W_kG_h}\vecP_{h,k}^-}{\vecm_{h,k}^-\times\vecvarphi}_{D_T}
 = O(h+k)
 \end{align}
 and
\begin{align}\label{wE:Phk}
\mu_0\inpro{\vecP_{h,k}}{\veczeta_t}_{\wtd D_T}
&-
\mu_0\inpro{\vecP_h^{(0)}}{\veczeta(0,\cdot)}_{\wtd D}
-
\inpro{\sig\nabla\times\vecP_{h,k}^+}{\nabla\times\veczeta}_{
\wtd D_T}\nn\\
&+
\sig_D\inpro{e^{W_kG_h}\vecm_{h,k}^-}{\nabla \times(\nabla\times\veczeta)}_{
D_T}
=O(h+k).
\end{align}
 \end{lemma}
\begin{proof}
\mbox{}

\noindent
\underline{Proof of~\eqref{InE:10}:}
For $t\in[t_j,t_{j+1})$, we use~\eqref{disE:LLG} with 
$\vecw_h^{(j)}=I_{\mV_h}\big(\vecm_{h,k}^-(t,\cdot)\times\vecvarphi(t,\cdot)\big)
\in\mW_h^{(j)}$
to have
\begin{align*}
-\lambda_1
&\inpro{\vecm_{h,k}^-(t,\cdot)\times\vecv_{h,k}(t,\cdot)}
{I_{\mV_h}\big(\vecm_{h,k}^-(t,\cdot)\times\vecvarphi(t,\cdot)\big)}_D \\
&+
\lambda_2\inpro{\vecv_{h,k}(t,\cdot)}
{I_{\mV_h}\big(\vecm_{h,k}^-(t,\cdot)\times\vecvarphi(t,\cdot)\big)}_D\nonumber\\
&+
\mu\inpro{\nabla(\vecm_{h,k}^-(t,\cdot)+k\theta\vecv_{h,k}(t,\cdot))}
{\nabla
I_{\mV_h}\big(\vecm_{h,k}^-(t,\cdot)\times\vecvarphi(t,\cdot)\big)}_D\\
&+
\inpro{R_{h,k}(t_j,\vecm_{h,k}^-(t,\cdot))}
{I_{\mV_h}\big(\vecm_{h,k}^-(t,\cdot)\times\vecvarphi(t,\cdot)\big)}_D\\
&-\mu
\inpro{e^{W_k(t)G_h}\vecP_{h,k}^-(t)}%
{I_{\mV_h}\big(\vecm_{h,k}^-(t,\cdot)\times\vecvarphi(t,\cdot)\big)}_{D}
= 0.
\end{align*}
Integrating both sides of the above equation over $(t_j,t_{j+1})$ and summing
over $j=0,\cdots,J-1$
we deduce
\begin{align*}
&-\lambda_1
\inpro{\vecm_{h,k}^-\times\vecv_{h,k}}
{I_{\mV_h}\big(\vecm_{h,k}^-\times\vecvarphi\big)}_{D_T}
+
\lambda_2\inpro{\vecv_{h,k}}
{I_{\mV_h}\big(\vecm_{h,k}^-\times\vecvarphi\big)}_{D_T}\nonumber\\
&+
\mu\inpro{\nabla(\vecm_{h,k}^-+k\theta\vecv_{h,k})}
{\nabla I_{\mV_h}\big(\vecm_{h,k}^-\times\vecvarphi\big)}_{D_T}
+
\inpro{R_{h,k}(\cdot,\vecm_{h,k}^-)}
{I_{\mV_h}\big(\vecm_{h,k}^-\times\vecvarphi\big)}_{D_T}\nn\\
&-\mu
\inpro{e^{W_kG_h}\vecP_{h,k}^-}{I_{\mV_h}\big(\vecm_{h,k}^-\times\vecvarphi\big)}_{D_T}
= 0.
\end{align*}
This implies
\begin{align*}
&-\lambda_1\inpro{\vecm_{h,k}^-\times\vecv_{h,k}}
{\vecm_{h,k}^-\times\vecvarphi}_{D_T}
+
\lambda_2\inpro{\vecv_{h,k}}
{\vecm_{h,k}^-\times\vecvarphi}_{D_T}\nonumber\\
&+
\mu\inpro{\nabla(\vecm_{h,k}^-+k\theta\vecv_{h,k})}
{\nabla(\vecm_{h,k}^-\times\vecvarphi)}_{D_T}
+
\inpro{R_{h,k}(\cdot,\vecm_{h,k}^-)}
{\vecm_{h,k}^-\times\vecvarphi}_{D_T}\nn\\
&-\mu
\inpro{e^{W_kG_h}\vecP_{h,k}^-}{\vecm_{h,k}^-\times\vecvarphi}_{D_T}
=I_1+I_2+I_3+I_4,
\end{align*}
where
\begin{align*}
I_1
&=
\inpro{-\lambda_1\vecm_{h,k}^-\times\vecv_{h,k}
+
\lambda_2\vecv_{h,k}}
{\vecm_{h,k}^-\times\vecvarphi
-
I_{\mV_h}(\vecm_{h,k}^-\times\vecvarphi)}_{D_T},\\
I_2
&=
\mu\inpro{\nabla(\vecm_{h,k}^-+k\theta\vecv_{h,k})}
{\nabla\big(\vecm_{h,k}^-\times\vecvarphi-
I_{\mV_h}(\vecm_{h,k}^-\times\vecvarphi)\big)}_{D_T},\\
I_3
&=
\inpro{R_{h,k}(.,\vecm_{h,k}^-)}
{\vecm_{h,k}^-\times\vecvarphi
-
I_{\mV_h}(\vecm_{h,k}^-\times\vecvarphi)}_{D_T},\\
I_4&=
-\inpro{e^{W_kG_h}\vecP_{h,k}^-}
{\vecm_{h,k}^-\times\vecvarphi-I_{\mV_h}(\vecm_{h,k}^-\times\vecvarphi))}_{D_T}.
\end{align*}
Hence it suffices to prove that $I_i=O(h+k)$ for $i=1,\cdots,4$.
Firstly, by using Lemma~\ref{lem:mhj} we obtain
\[
\norm{\vecm_{h,k}^-}{\mL^\infty(D_T)}
\le
\sup_{0\le j \le J} \norm{\vecm_h^{(j)}}{\mL^\infty(D)}
\leq
1.
\]
This inequality, Lemma~\ref{lem:3.2a} and Lemma
~\ref{lem:Ih vh} yield
\begin{align*}
|I_1|
&\le
c\left(\norm{\vecm_{h,k}^-}{\mL^\infty(D_T)}+1\right)
\norm{\vecv_{h,k}}{D_T}
\norm{\vecm_{h,k}^-\times\vecvarphi
-
I_{\mV_h}(\vecm_{h,k}^-\times\vecvarphi)}{D_T} \\
&\le
c
\norm{\vecm_{h,k}^-\times\vecvarphi
-
I_{\mV_h}(\vecm_{h,k}^-\times\vecvarphi)}{D_T} 
\le
ch.
\end{align*}
The bounds for $I_2$, $I_3$ and $I_4$ can be obtained similarly
by using Lemma~\ref{lem:3.2a} and Lemma~\ref{lem:Fk},
respectively, noting that when $\theta\in[0,\frac{1}{2}]$, a bound of
$\left\| \nabla \vecv_{h,k} \right\|_{D_T}$ can be deduced from the inverse
estimate
$
\left\| \nabla \vecv_{h,k} \right\|_{D_T}
\leq
ch^{-1}
\left \| \vecv_{h,k}\right\|_{D_T}.
$
This completes the proof~\eqref{InE:10}.

\medskip
\noindent
\underline{Proof of~\eqref{wE:Phk}:}
For $t\in[t_j,t_{j+1})$, we use~\eqref{disE:NewMax2} with 
$\veczeta_h(t,\cdot)=I_{\mY_h}\veczeta(t,\cdot)$ to have 
\begin{align*}
\mu_0\inpro{\partial_t\vecP_{h,k}(t,\cdot)}{I_{\mY_h}
\veczeta(t,\cdot)}_{\wtd D}
=
&-
\inpro{\sigma\nabla\times\vecP_{h,k}^+(t,\cdot)}
{\nabla\times I_{\mY_h}\veczeta(t,\cdot)}_{\wtd D}\\
&+
\sigma_D
\inpro{\nabla\times e^{W_k(t)G_h}\vecm_{h,k}^-(t,\cdot)}
{\nabla\times I_{\mY_h}\veczeta(t,\cdot)}_{D}.
\end{align*}
Integrating both sides of the above equation over $(t_j,t_{j+1})$ 
andsumming over $j=0,\cdots,J-1$, and using integration by parts (noting that
$\veczeta_h(T,\cdot)=0$) we deduce
\begin{align*}
\mu_0\inpro{\vecP_{h,k}}{\partial_t \veczeta_h}_{\wtd D_T}
-
\mu_0\inpro{\vecP_h^{(0)}}{\veczeta_h(0,\cdot)}_{\wtd D}
=
&
\inpro{\sigma\nabla\times\vecP_{h,k}^+}
{\nabla\times \veczeta_h}_{\wtd D_T}\\
&-
\sigma_D
\inpro{\nabla\times e^{W_kG_h}\vecm_{h,k}^-}
{\nabla\times \veczeta_h}_{D_T}.
\end{align*}
By using Lemma~\ref{lem:3.2b} and the following error estimate, see 
e.g.~\cite{Monk03},
\begin{equation*}
\|\veczeta(t)-\veczeta_h(t)\|_{\wtd D}
+
h\|\nabla\times(\veczeta(t)-\veczeta_h(t))\|_{\wtd D}
\le
Ch^2
\|\nabla^2\veczeta\|_{\wtd D},\label{InE:IterY}
\end{equation*}
we deduce
\begin{align*}
\mu_0\inpro{\vecP_{h,k}}{
\veczeta_t)}_{\wtd D_T}
-
\mu_0\inpro{\vecP_h^{(0)}}{
\veczeta(0,\cdot)}_{\wtd D}
&-
\inpro{\sigma\nabla\times\vecP_{h,k}^+}
{\nabla\times\veczeta}_{\wtd D_T}\\
&+
\sigma_D
\inpro{\nabla\times e^{W_kG_h}\vecm_{h,k}^-}
{\nabla\times\veczeta}_{D_T}
= O(h).
\end{align*}
Using Green's identity (see~\cite[Corollary 3.20]{Monk03})
we obtain~\eqref{wE:Phk}, completing the proof of the lemma. 
\end{proof}

In the next lemma we show that $\vecv_{h,k}$ can be replaced by
$\pa_t\vecm_{h,k}$, as indeed the latter approximates~$\vecm_t$.
\begin{lemma}\label{lem:3.7b}
 Assume that $h$ and $k$ go to 0 satisfying~\eqref{equ:theta}. Then for 
 any $\vecvarphi \in \CoE{C^\infty(D)}$ and 
$\veczeta \in \CTE{\CX}$, there holds $\mP$-a.s.
 \begin{align}\label{InE:10b}
 &-\lambda_1\inpro{\vecm_{h,k}\times\pa_t\vecm_{h,k}}
 {\vecm_{h,k}\times\vecvarphi}_{D_T}
 +
 \lambda_2\inpro{\pa_t\vecm_{h,k}}
 {\vecm_{h,k}\times\vecvarphi}_{D_T}\nn\\
 &+
 \mu\inpro{\nabla\vecm_{h,k}}
 {\nabla(\vecm_{h,k}\times\vecvarphi)}_{D_T}
 +
 \inpro{R_{h,k}(\cdot,\vecm_{h,k})}
 {\vecm_{h,k}\times\vecvarphi}_{D_T}\nn\\
 &-\mu
\inpro{e^{W_kG_h}\vecP_{h,k}^+}{\vecm_{h,k}\times\vecvarphi}_{D_T}
 = O(h+k);
 \end{align}
 and
\begin{align}\label{wE:Phkb}
\mu_0\inpro{\vecP_{h,k}^+}{\veczeta_t}_{\wtd D_T}
-
\mu_0\inpro{\vecP_h^{(0)}}{\veczeta(0,\cdot)}_{\wtd D}
&-
\inpro{\sig\nabla\times\vecP_{h,k}^+}{\nabla\times\veczeta}_{\wtd D_T}
\nn\\
&+
\sig_D\inpro{e^{W_kG_h}\vecm_{h,k}}{\nabla \times(\nabla\times\veczeta)}_{D_T}
=O(h+k).
\end{align}
 \end{lemma}
\begin{proof}
\mbox{}

\noindent
\underline{Proof of~\eqref{InE:10b}:}
From~\eqref{InE:10} it follows that
\begin{align*}
&-\lambda_1\inpro{\vecm_{h,k}\times\pa_t\vecm_{h,k}}
 {\vecm_{h,k}\times\vecvarphi}_{D_T}
 +
 \lambda_2\inpro{\pa_t\vecm_{h,k}}
 {\vecm_{h,k}\times\vecvarphi}_{D_T}\nn\\
 &+
 \mu\inpro{\nabla(\vecm_{h,k})}
 {\nabla(\vecm_{h,k}\times\vecvarphi)}_{D_T}
 +
 \inpro{R_{h,k}(\cdot,\vecm_{h,k})}
 {\vecm_{h,k}\times\vecvarphi}_{D_T}\nn\\
 &-\mu
\inpro{e^{W_kG_h}\vecP_{h,k}^+}{\vecm_{h,k}\times\vecvarphi}_{D_T}
 = 
 O(h+k)
 +
I_1+\cdots+I_5,
\end{align*}
where
\begin{align*}
I_1
&=
\lambda_1\inpro{\vecm_{h,k}^-\times\vecv_{h,k}}
{\vecm_{h,k}^-\times\vecvarphi}_{D_T}
-
\lambda_1\inpro{\vecm_{h,k}\times\pa_t\vecm_{h,k}}
{\vecm_{h,k}\times\vecvarphi}_{D_T},\\
I_2
&=
-\lambda_2\inpro{\vecv_{h,k}}
{\vecm_{h,k}^-\times\vecvarphi}_{D_T}
+
\lambda_2\inpro{\pa_t\vecm_{h,k}}
{\vecm_{h,k}\times\vecvarphi}_{D_T}
,\\
I_3
&=
-\mu\inpro{\nabla(\vecm_{h,k}^-+k\theta\vecv_{h,k})}
{\nabla(\vecm_{h,k}^-\times\vecvarphi)}_{D_T}
+
\mu\inpro{\nabla(\vecm_{h,k})}
{\nabla(\vecm_{h,k}\times\vecvarphi)}_{D_T}
,\\
I_4
&=
-\inpro{R_{h,k}(.,\vecm_{h,k}^-)}
{\vecm_{h,k}^-\times\vecvarphi}_{D_T}
+
\inpro{R_{h,k}(\cdot,\vecm_{h,k})}
{\vecm_{h,k}\times\vecvarphi}_{D_T},\\
I_5
&=
-\mu
\inpro{e^{W_kG_h}\vecP_{h,k}^+}{\vecm_{h,k}\times\vecvarphi}_{D_T}
+
\mu
\inpro{e^{W_kG_h}\vecP_{h,k}^-}{\vecm_{h,k}^-\times\vecvarphi}_{D_T}
.
\end{align*}
Hence it suffices to prove that $I_i=O(k)$ for $i=1,\cdots,5$. 

First, by using the triangle inequality we obtain
\begin{align*}
\lambda_1^{-1}|I_1|
&\leq
\left|
\inpro{(\vecm_{h,k}^--\vecm_{h,k})\times\vecv_{h,k}}
{\vecm_{h,k}^-\times\vecvarphi}_{D_T}
\right|
+
\left|
\inpro{\vecm_{h,k}\times\vecv_{h,k}}
{(\vecm_{h,k}^--\vecm_{h,k})\times\vecvarphi}_{D_T}
\right|\\
&\quad+
\left|
\inpro{\vecm_{h,k}\times(\vecv_{h,k}-\pa_t\vecm_{h,k})}
{\vecm_{h,k}\times\vecvarphi}_{D_T}
\right|\\
&\leq
2\norm{\vecm_{h,k}^--\vecm_{h,k}}{D_T}
\norm{\vecv_{h,k}}{D_T}
\norm{\vecm_{h,k}^-}{\mL^{\infty}(D_T)}
\norm{\vecvarphi}{\mL^{\infty}(D_T)}\\
&\quad+
\norm{\vecv_{h,k}-\pa_t\vecm_{h,k}}{\mL^1(D_T)}
\norm{\vecm_{h,k}^-}{\mL^{\infty}(D_T)}
\norm{\vecvarphi}{\mL^{\infty}(D_T)}.
\end{align*}
Therefore, the bound of $I_1$ can be obtained by using Lemmas~\ref{lem:3.2a}
and~\ref{lem:3.4}. The bounds for $I_2, I_3$ and $I_4$  can be obtained
similarly.

Finally, using~\eqref{equ:boundP}, Lemmas~\ref{lem:3.2a} and~\ref{lem:3.2b} 
we obtain
\begin{align*}
\mu^{-1}|I_5|
&\leq
\left|
\inpro{e^{W_kG_h}(\vecP_{h,k}^+-\vecP_{h,k}^-)}{\vecm_{h,k}\times\vecvarphi}_{D_T}
\right|
+
\left|
\inpro{e^{W_kG_h}\vecP_{h,k}^-}{(\vecm_{h,k}-\vecm_{h,k}^-)\times\vecvarphi}_{D_T}
\right|\\
&\leq
\norm{e^{W_kG_h}(\vecP_{h,k}^+-\vecP_{h,k}^-)}{D_T}
\norm{\vecm_{h,k}}{D_T}
\norm{\vecvarphi}{\mL^{\infty}(D_T)}\\
&\quad
+
\norm{e^{W_kG_h}\vecP_{h,k}^-}{D_T}
\norm{\vecm_{h,k}-\vecm_{h,k}^-}{D_T}
\norm{\vecvarphi}{\mL^{\infty}(D_T)}\\
&\leq
c\norm{\vecP_{h,k}^+-\vecP_{h,k}^-}{D_T}
+
c\norm{\vecP_{h,k}^-}{D_T}
\norm{\vecm_{h,k}-\vecm_{h,k}^-}{D_T}
\leq 
ck.
\end{align*}
This completes the proof of~\eqref{InE:10b}.

\medskip
\noindent
\underline{Proof of~\eqref{wE:Phkb}:}
It follows from~\eqref{wE:Phk} that
\begin{align*}
\mu_0\inpro{\vecP_{h,k}^+}{\veczeta_t}_{\wtd D_T}
&-
\mu_0\inpro{\vecP_h^{(0)}}{\veczeta(0,\cdot)}_{\wtd D}
-
\inpro{\sig\nabla\times\vecP_{h,k}^+}{\nabla\times\veczeta}_{\wtd D_T}\nn\\
&+
\sig_D\inpro{e^{W_kG_h}\vecm_{h,k}}{\nabla \times(\nabla\times\veczeta)}_{D_T}
=O(h+k)
+I_6+I_7,
\end{align*}
where
\begin{align*}
I_6
&=
\mu_0\inpro{\vecP_{h,k}^+-\vecP_{h,k}}{\veczeta_t}_{\wtd
D_T},
\nn\\
I_7
&=
\sig\inpro{e^{W_kG_h}(\vecm_{h,k}^--\vecm_{h,k})}%
{\nabla \times(\nabla\times\veczeta)}_{D_T}.
\end{align*}
By using~\eqref{equ:mhk mhkm} and~\eqref{equ:Phk2} we obtain that $I_i=O(k)$
for $i=6,7$.
This completes the proof of~\eqref{wE:Phkb}.
\end{proof}
In order to prove the $\mP$-a.s. convergence of random variables 
$\vecm_{h,k}$ and $\vecP_{h,k}^+$, we first show that the
family $\cL(\vecm_{h,k})$ and $\cL(\vecP_{h,k}^+)$ are tight. 
\begin{lemma}\label{lem:tig}
Assume that $h$ and $k$ go to 0 satisfying~\eqref{equ:theta}.
Then the set of laws 
$\{\cL(\vecm_{h,k},\vecP_{h,k}^+,W_k)\}$ on the space
$C\big(0,T;\mH^{-1}(D)\big) \times H^{-1}(\wtd D_T) \times \D(0,T)$ 
is tight. Here, $\D(0,T)$ is the Skorokhod space; see e.g.~\cite{Bill99}.
\end{lemma}
\begin{proof}
Firstly,
from Definition~\ref{Def:W}, the approximation $W_k$ of the Wiener process 
 $W$ belongs to~$\D(0,T)$.
The tightness of~$\{\cL(W_k)\}$ in $\D(0,T)$ is proved 
in~\cite[Theorem~2.5.6]{Bill99}.
The tightness of $\{\cL(\vecm_{h,k})\}$ on $C\big(0,T;\mH^{-1}(D)\big)$ 
and of $\{\cL(\vecP_{h,k}^+)\}$ on $ H^{-1}(\wtd D_T)$ can be obtained as in the
proof of \cite[Lemma 6.6]{BNT13} and is therefore omitted.
\end{proof}
 The following proposition is a consequence of the tightness of 
 $\{\cL(\vecm_{h,k})\}$, $\{\cL(\vecP_{h,k}^+)\}$ and  $\{\cL(W_k)\}$.
\begin{proposition}\label{pro:con}
Assume that $h$ and $k$ go to 0 satisfying~\eqref{equ:theta}.
Then there exist
\begin{enumerate}
\renewcommand{\labelenumi}{(\alph{enumi})}
\item
a probability space $(\Omega',\cF',\mP')$,
\item
a sequence $\{(\vecm'_{h,k},\vecP'_{h,k},W_k')\}$ of random variables
defined on $(\Omega',\cF',\mP')$ and taking values in the space 
$C\big(0,T;\mH^{-1}(D)\big)\times \mH^{-1}(\wtd D_T) \times\D(0,T)$, 
\item
a random variable $(\vecm',\vecP',W')$ defined on
$(\Omega',\cF',\mP')$ and taking values in 
$C\big([0,T];\mH^{-1}(D)\big)\times \mH^{-1}(\wtd D_T) \times \D(0,T)$,
\end{enumerate}
satisfying
\begin{enumerate}
\item\label{item:a}
$\cL(\vecm_{h,k},\vecP_{h,k}^+,W_k) = \cL(\vecm_{h,k}',\vecP'_{h,k},W_k')$,
\item\label{item:b}
$\vecm_{h,k}'\goto\vecm'$ in $C\big(0,T;\mH^{-1}(D)\big)$ strongly,
$\mP'$-a.s.,
\item\label{item:b2}
$\vecP_{h,k}'\goto\vecP'$ in $\mH^{-1}(\wtd D_T)$ strongly,
$\mP'$-a.s.,
\item\label{item:c}
$W_k'\goto W'$ in $\D(0,T)$ $\mP'$-a.s.
\end{enumerate}

Moreover, the sequence $\{\vecm'_{h,k}\}$ and $\{\vecP'_{h,k}\}$ satisfy $\mathbb P^\prime$-a.s. 
\begin{align}
\norm{\vecm_{h,k}'(\omega^\prime)}{\mH^1(D_T)}
&\leq c, \label{equ:mhk' h1} \\
\norm{\vecm_{h,k}'(\omega^\prime)}{\mL^{\infty}(D_T)}
&\leq c, 
\label{mhk'Linf} \\
\norm{|\vecm_{h,k}'(\omega^\prime)|-1}{\mL^2(D_T)} &\leq c(h+k),\label{equ:mhk'1}\\
\text{and}\quad 
\norm{\vecP'_{h,k}(\omega)}{L^2(0,T;\mH(\curl;\wtd D))}
&\leq c.\label{equ:Phk'}
\end{align}
\end{proposition}

\begin{proof}
By Lemma~\ref{lem:tig} and the Donsker theorem~\cite[Theorem 8.2]{Bill99},
the family of probability measures $\{\cL(\vecm_{h,k},\vecP_{h,k}^+,W_k)\}$ is tight on 
$C\big(0,T;\mH^{-1}(D)\big)\times \mH^{-1}(\wtd D_T) \times\D(0,T)$. 
Then by Theorem 5.1 in~\cite{Bill99} the family of measures $\{\cL(\vecm_{h,k},\vecP_{h,k}^+,W_k)\}$ 
is relatively compact on $C\big(0,T;\mH^{-1}(D)\big)\times \mH^{-1}(\wtd D_T) \times\D(0,T)$, 
that is there exists a subsequence, still denoted by $\{\cL(\vecm_{h,k},\vecP_{h,k}^+,W_k)\}$, 
such that $\{\cL(\vecm_{h,k},\vecP_{h,k}^+,W_k)\}$ converges weakly. 
Hence, the existence of (a)--(c) satisfying (1)--(4) follows immediately from the 
Skorokhod Theorem~\cite[Theorem 6.7]{Bill99} since $C\big(0,T;\mH^{-1}(D)\big)\times \mH^{-1}(\wtd D_T) \times\D(0,T)$ 
is a separable metric space.

We note that from the Kuratowski theorem, the Borel subsets of $\mH^1(D_T)$ or $\mH^1(D_T)\cap \mL^{\infty}(D_T)$ 
are Borel subsets of $C\big(0,T;\mH^{-1}(D)\big)$ and the Borel subsets of $L^2(0,T;\mH(\curl;\wtd D))$ are 
Borel subsets of $\mH^{-1}(\wtd D_T)$. 
The estimates~\eqref{equ:mhk' h1}--\eqref{equ:Phk'} are direct consequences of Lemmas~\ref{lem:3.2b}--\ref{lem:3.4} and 
the equality of laws stated in part (1).
\end{proof}
 We now ready to prove the main result of this paper.
\begin{theorem}\label{the:mai 2}
Assume that $T>0$, $\vecM_0\in\mH^2(D)$ and $\vecg\in\mW^{2,\infty}(D)$
satisfy~\eqref{equ:m0} and \eqref{equ:g 1}, respectively. 
Then $\vecm'$, $\vecP'$, $W'$, the sequences $\{\vecm'_{h,k}\}$, $\{\vecP'_{h,k}\}$ 
and the probability space $(\Omega',\cF',\mP')$  given by Proposition~\ref{pro:con} satisfy
\begin{enumerate}
\item\label{ite:the1}
the sequence $\{\vecm'_{h,k}\}$ converges to $\vecm'$
weakly in $\mH^1(D_T)$, $\mP'$-a.s.
\item\label{ite:the3}
the sequence $\{\vecP'_{h,k}\}$ converges to $\vecP'$
weakly in $L^2(0,T;\mH(\curl; \wtd D)$, $\mP'$-a.s.
\item\label{ite:the2}
$\big(\Omega',\cF',(\cF'_t)_{t\in[0,T]},\mP',W',\vecM',\vecP'\big)$ is a weak
martingale solution of~\eqref{E:1.1},
where 
\[
\vecM'(t):=e^{W'(t)G}\vecm'(t)\quad \forall t\in[0,T],\text{ a.e. } \vecx\in D.
\] 
\end{enumerate}
\end{theorem}
 
\begin{proof}
By Proposition~\ref{pro:con} there exists a set $V\subset\Omega'$ such that 
$\mP'(V) = 1$,
\[
\vecm_{h,k}'(\omega')\goto\vecm'(\omega')\quad\text{strongly in}\quad C\big(0,T;\mH^{-1}(D)\big), 
\]
\[
\vecP_{h,k}'(\omega')\goto\vecP'(\omega')\quad\text{strongly in}\quad\mH^{-1}(\wtd D_T), 
\]
and~\eqref{equ:mhk' h1},~\eqref{equ:Phk'} hold for every $\omega'\in V$.  
In what follows, we work with a fixed $\omega'\in V$. 

The convergences of sequences $\{\vecm'_{h,k}(\omega')\}$  and $\{\vecP'_{h,k}(\omega')\}$ 
are obtained by using the same arguments as in~\cite[Theorem 6.8]{GoldysLeTran2016}.

In order to prove~\eqref{ite:the2}, by noting Lemma~\ref{lem:equi} 
we need to prove that $\vecm',\vecP'$ and $W'$ satisfy~\eqref{equ:m 1},~\eqref{InE:13} 
and~\eqref{wE:vecP}.

\underline{Prove that $\vecm'$ satisfies~\eqref{equ:m 1}:}
Since $\mH^1(D_T)$ is compactly embedded in $\mL^2(D_T)$, 
there exists a subsequence of $\{\vecm_{h,k}'(\omega')\}$ (still denoted by $\{\vecm_{h,k}'(\omega')\}$) 
such that 
\begin{equation}\label{equ:strconvem'}
 \vecm_{h,k}'(\omega')\goto\vecm'(\omega')\quad\text{strongly in}\quad \mL^2(D_T).
\end{equation}
Therefore~\eqref{equ:m 1} follows from~\eqref{equ:strconvem'} and~\eqref{equ:mhk'1}.

\underline{Prove that $\vecm',\vecP'$ satisfy~\eqref{InE:13} and~\eqref{wE:vecP}:}
From Lemma~\ref{lem:3.7b}
 , $\big(\vecm_{h,k},\vecP_{h,k}^+,W_k\big)$
 satisfies~\eqref{InE:10b}--\eqref{wE:Phkb} $\mP$-a.s.. 
 Therefore, it follows from the equality of laws in Proposition~\ref{pro:con} that 
 $\big(\vecm_{h,k}',\vecP_{h,k}',W_k'\big)$ satisfies the
following equations for all $\vecpsi\in
C_0^{\infty}\big((0,T);\C^{\infty}(D)\big)$ and $\veczeta \in
C^{\infty}_c([0, T),\CX)$, $\mP'$-a.s.
 \begin{align}\label{equ:mhk'}
 &-\lambda_1\inpro{\vecm_{h,k}'(\omega')\times\pa_t\vecm_{h,k}'(\omega')}
 {\vecm_{h,k}'(\omega')\times\vecpsi}_{D_T}
 +
 \lambda_2\inpro{\pa_t\vecm_{h,k}'(\omega')}
 {\vecm_{h,k}'(\omega')\times\vecpsi}_{D_T}\nn\\
 &+
 \mu\inpro{\nabla(\vecm_{h,k}'(\omega'))}
 {\nabla(\vecm_{h,k}'(\omega')\times\vecpsi)}_{D_T}
 +
 \inpro{R_{h,k}(\cdot,\vecm_{h,k}'(\omega'))}
 {\vecm_{h,k}'(\omega')\times\vecpsi}_{D_T}\nn\\
 &\quad\quad-\mu
 \inpro{e^{W'_kG_h}\vecP_{h,k}'(\omega')}{\vecm_{h,k}'(\omega')\times\vecpsi}_{D_T}
 = O(h+k),
 \end{align}
 and
\begin{align}\label{wE:Phk'}
\mu_0\inpro{\vecP_{h,k}'(\omega')}{\veczeta_t}_{\wtd D_T}
&-
\mu_0\inpro{\vecP_h^{(0)}}{\veczeta(0,\cdot)}_{\wtd D_T}
+
\inpro{\sig\nabla\times\vecP_{h,k}'(\omega')}{\nabla\times\veczeta}_{\wtd D_T}\nn\\
&-
\sig_D
\inpro{ e^{W'_kG_h}\vecm_{h,k}'(\omega')}{\nabla \times (\nabla\times\veczeta)}_{D_T}
=O(h+k).
\end{align}

It suffices now to use the same arguments as in~\cite[Theorem 6.8]{GoldysLeTran2016} to 
pass the limit in~\eqref{equ:mhk'} and~\eqref{wE:Phk'}. Indeed, from ~\cite[Theorem 6.8]{GoldysLeTran2016} 
there hold
\begin{align*}
\inpro{\vecm'(\omega')\times\pa_t\vecm'(\omega')}
 {\vecm'(\omega')\times\vecpsi}_{D_T}
 &\goto 
 \inpro{\vecm'(\omega')\times\pa_t\vecm'(\omega')}
 {\vecm'(\omega')\times\vecpsi}_{D_T}\\
 \inpro{\pa_t\vecm_{h,k}'(\omega')}
 {\vecm_{h,k}'(\omega')\times\vecpsi}_{D_T}
 &\goto
 \inpro{\pa_t\vecm'(\omega')}
 {\vecm'(\omega')\times\vecpsi}_{D_T}\\
 \inpro{\nabla(\vecm_{h,k}'(\omega'))}
 {\nabla(\vecm_{h,k}'(\omega')\times\vecpsi)}_{D_T}
 &\goto
 \inpro{\nabla(\vecm'(\omega'))}
 {\nabla(\vecm'(\omega')\times\vecpsi)}_{D_T}\\
 \inpro{R_{h,k}(\cdot,\vecm_{h,k}'(\omega'))}
 {\vecm_{h,k}'(\omega')\times\vecpsi}_{D_T}
 &\goto
 \inpro{R(\cdot,\vecm'(\omega'))}
 {\vecm'(\omega')\times\vecpsi}_{D_T}.
\end{align*}
To prove the convergence of the last term in~\eqref{equ:mhk'}, 
we use the triangle inequality, H\"older inequaliy,~\eqref{equ:boundP},~\eqref{def:eh} 
and~\eqref{equ:G8} to obtain
\begin{align*}
\cI
&:=
\bigl|
\inpro{e^{W'_kG_h}\vecP_{h,k}'(\omega')}{\vecm_{h,k}'(\omega')\times\vecpsi}_{D_T}
-
\inpro{e^{W'G}\vecP'(\omega')}{\vecm'(\omega')\times\vecpsi}_{D_T}
\bigr|\\
&\leq \quad
\bigl|
\inpro{e^{W'_kG_h}\vecP_{h,k}'(\omega')}{\bigl(\vecm_{h,k}'(\omega')-\vecm'(\omega')\bigr)\times\vecpsi}_{D_T}
\bigr|\\
&\quad+
\bigl|
\inpro{(e^{W_k'G_h}-e^{W_k'G})\vecP_{h,k}'(\omega')}{\vecm'(\omega')\times\vecpsi}_{D_T}
\bigr|\\
&\quad+
\bigl|
\inpro{(e^{W_k'G}-e^{W'G})\vecP_{h,k}'(\omega')}{\vecm'(\omega')\times\vecpsi}_{D_T}
\bigr|\\
&\quad+
\bigl|
\inpro{e^{W'G}\bigl(\vecP_{h,k}'(\omega')-\vecP'(\omega')\bigr)}{\vecm'(\omega')\times\vecpsi}_{D_T}
\bigr|\\
&\leq \quad
\norm{\vecP_{h,k}'(\omega')}{D_T} 
\norm{\vecm_{h,k}'(\omega')-\vecm'(\omega')}{D_T} 
\norm{\vecpsi}{\mL^{\infty}(D_T)}\\
&\quad+
c\norm{I_{\mV_h}(\vecg)-\vecg}{D}
\norm{\vecP_{h,k}'(\omega')}{D_T} 
\norm{\vecm'(\omega')}{\mL^{\infty}(D_T)}
\norm{\vecpsi}{\mL^{\infty}(D_T)}\\
&\quad+
c\norm{W_k(\omega')-W'(\omega')}{\mL^{\infty}([0,T])}
\norm{\vecP_{h,k}'(\omega')}{D_T} 
\norm{\vecm'(\omega')}{\mL^{\infty}(D_T)}
\norm{\vecpsi}{\mL^{\infty}(D_T)}\\
&\quad+
\bigl|
\inpro{\vecP_{h,k}'(\omega')-\vecP'(\omega')}{e^{-W'G}\bigl(\vecm'(\omega')\times\vecpsi\bigr)}_{D_T}
\bigr|\\
&\leq 
c \norm{\vecm_{h,k}'(\omega')-\vecm'(\omega')}{D_T}
+
c\norm{I_{\mV_h}(\vecg)-\vecg}{D}
+c\norm{W_k(\omega')-W'(\omega')}{\mL^{\infty}([0,T])}\\
&\quad+
\bigl|
\inpro{\vecP_{h,k}'(\omega')-\vecP'(\omega')}{e^{-W'G}\bigl(\vecm'(\omega')\times\vecpsi\bigr)}_{D_T}
\bigr|,
\end{align*}
here the last inequality is obtained by using~\eqref{equ:Phk'} and $|\vecm'(\omega')|=1$ a.e..

Hence, it follows from~\eqref{equ:strconvem'}, part (4) in Proposition~\ref{pro:con} 
and the weak convergence of $\{\vecP'_{h,k}(\omega')\}$ in $L^2(0,T;\mH(\curl;\wtd D))$ that 
\begin{equation*}
 \cI\goto 0 \quad \text{as } h,k\goto 0.
\end{equation*}
This implies that $\vecm',\vecP'$ satisfy~\eqref{InE:13}.

The convergence of~\eqref{wE:Phk'} can be proved in the same manner by noting that
$\{\vecP'_{h,k}(\omega')\}$ converges weakly in $L^2(0,T;\mH(\curl;\wtd D))$, 
completing the proof of the theorem.
\end{proof}
\section{Numerical experiment}\label{sec:num}
In order to carry out physically relevant experiments (see~\cite{GuoDing08}),
the initial fields $\vecM_0$, $\vecH_0$ must satisfy 
the following conditions
\begin{equation*}\label{E:Cond2}
\dive (\vecH_0+\wtd\vecM_0)=0\text{ in } \wtd D
\quad\text{and}\quad
(\vecH_0+\wtd\vecM_0)\cdot\vecn = 0\text{ on }\partial \wtd  D.
\end{equation*}
This can be achieved by taking
\begin{equation*}
\vecH_0 =\vecH_0^* -  \chi_D\vecM_0,
\end{equation*}
where $\dive \vecH_0^* = 0$ in $\wtd D$. 
In our experiment, for simplicity, we choose
$\vecH_0^*$ to be a constant.
We solve an academic example with $D=\wtd D=(0,1)^3$ and 
\begin{align*}
\vecM_0(\vecx)
&=
\begin{cases}
(0,0,-1), &\quad |\vecx^*|\geq \frac{1}{2},\\
(2\vecx^*A,A^2-|\vecx^*|^2)/(A^2+|\vecx^*|^2),
&\quad |\vecx^*|\leq \frac{1}{2},
\end{cases}\\
\vecH_0^*(\vecx)
&=
(0,0,H_s),\quad\vecx\in \wtd D,
\end{align*}
where $\vecx=(x_1,x_2,x_3)$, $\vecx^*=(x_1-0.5,x_2-0.5,0)$ and $A=(1-2|\vecx^*|)^4/4$.
The constant $H_s$ represents the strength of $\vecH_0$
in the  $x_3$-direction. We carried out the experiments for
$H_s=30$.
We set the values for the other parameters
in~\eqref{wE:1.1} and~\eqref{wE:Maxwell2} as
$\lam_1=\lam_2=\mu_0=\sig=1$.

For each time step $k$, we generate a discrete  Brownian path by:
\[
W_k(t_{j+1})-W_k(t_{j})\sim \cN(0,k)\quad \text{for all
}j=0,\cdots,J-1.
\]
An approximation of any expected value is computed as
the average of $L$
discrete Brownian paths. In our experiments, we choose  $L=400$. 

At each iteration we solve two linear systems of sizes $2N\times2N$ and $M\times M$,
recalling that $N$ is the number of vertices and $M$ is
the number of edges in the triangulation. The code is
written in Fortran90.
The parameter $\theta$ in Algorithm~\ref{Algo:1} is chosen
to be $0.7$.

In the first set of experiments, to observe
convergence of the method,
we solve with $T=1$, $h=1/n$ where $n=2,\ldots,7$,
and different time steps $k=h$, $k=h/2$, and $k=h/4$. 
For each value of $h$,
the domain $D$ is partitioned into uniform cubes of
size $h$. Each cube is then partitioned into
six tetrahedra. 
Noting that
\[
E_{h,k}^2 :=
\int_{D_T}\left|1-|\vecm_{h,k}^-|\right|^2\dvx\dt
=
\norm{|\vecm|-|\vecm_{h,k}^-|}{D_T}^2
\le
\norm{\vecm-\vecm_{h,k}^-}{D_T}^2,
\]
we compute and plotte in
Figure~\ref{fig:error} the error $\E[E_{h,k}^2]$ for different values of $h$ and
$k$. 
\begin{figure}
\begin{center}
\includegraphics[width=15cm,height=9cm]{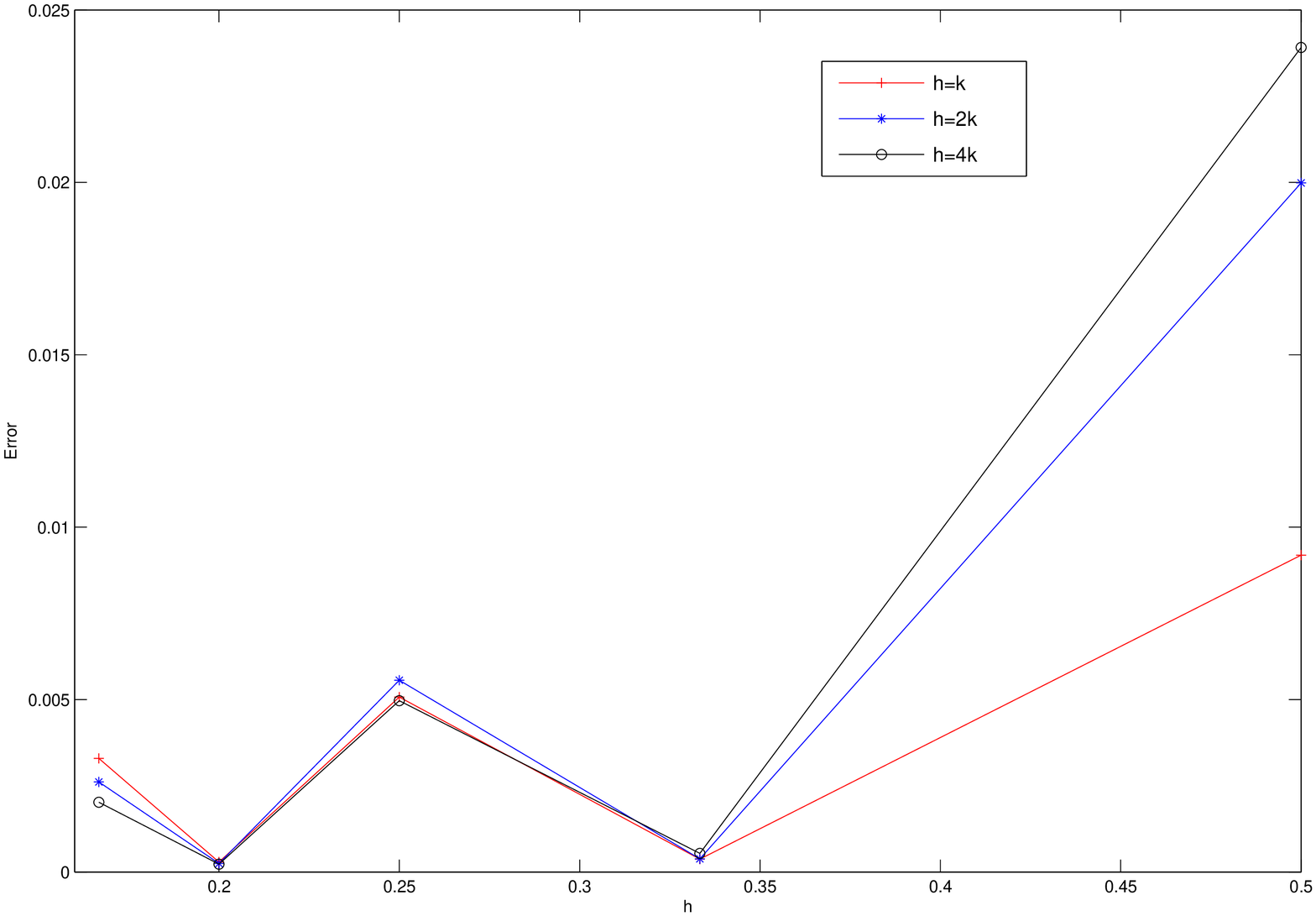}
\caption{Plot of error $\E[E_{h,k}^2]$}
\label{fig:error}
\end{center}
\end{figure}

In the second set of experiments to observe boundedness of
discrete energies, we solve the problem with fixed values
of $h=1/7$ and $k=1/20$. 
We
plot~$t\mapsto\|\nabla\vecm_{h,k}(t)\|_{D}^2$ in Figure~\ref{Fig:Eex} and 
 $t\mapsto\|\vecP_{h,k}(t)\|_{\wtd D}^2$ in Figure \ref{Fig:Eh} for three
individual paths and the expectations which seems to suggest that these
energies are bounded when $t\goto\infty$. 
Figure~\ref{fig:Etotal} shows that the total energy
$\cE(t):=\|\nabla\vecm_{h,k}(t)\|_{D}^2+\|\vecP_{h,k}(t)\|_{\wtd D}$ is
bounded as in Lemma~\ref{lem:4.2}.
\begin{figure}
\begin{center}
\includegraphics[width=15cm,height=9cm]{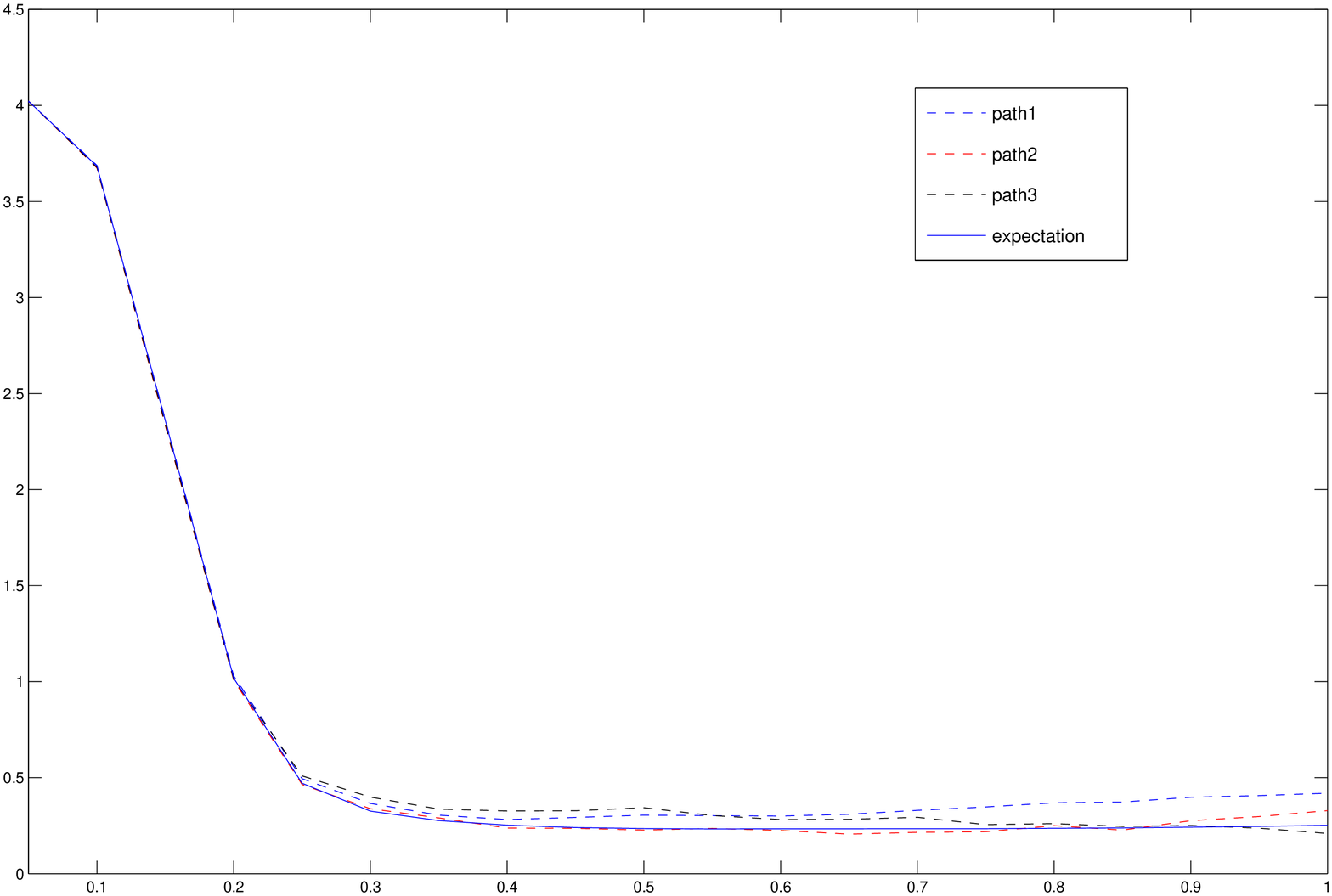}
\caption{Plot of $t\mapsto\|\nabla\vecM_{h,k}(t)\|_D$, expectation and three
individual paths}\label{Fig:Eex}
\end{center}
\end{figure}
\begin{figure}
\begin{center}
\includegraphics[width=15cm,height=9cm]{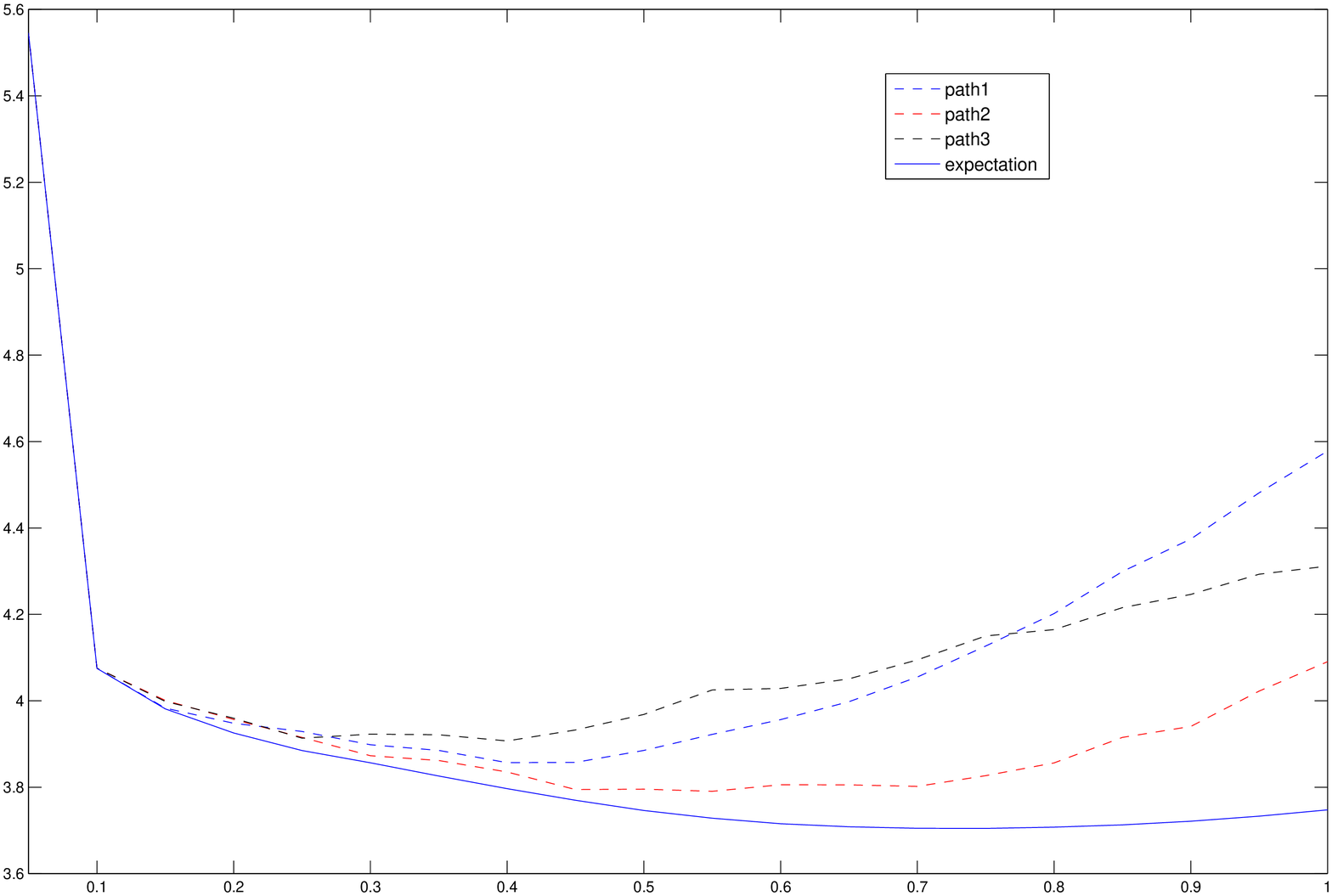}
\caption{Plot of $t\mapsto\|\vecP_{h,k}(t)\|_{\wtd D}$, expectation and three
individual paths}\label{Fig:Eh}
\end{center}
\end{figure}
\begin{figure}
\begin{center}
\includegraphics[width=15cm,height=9cm]{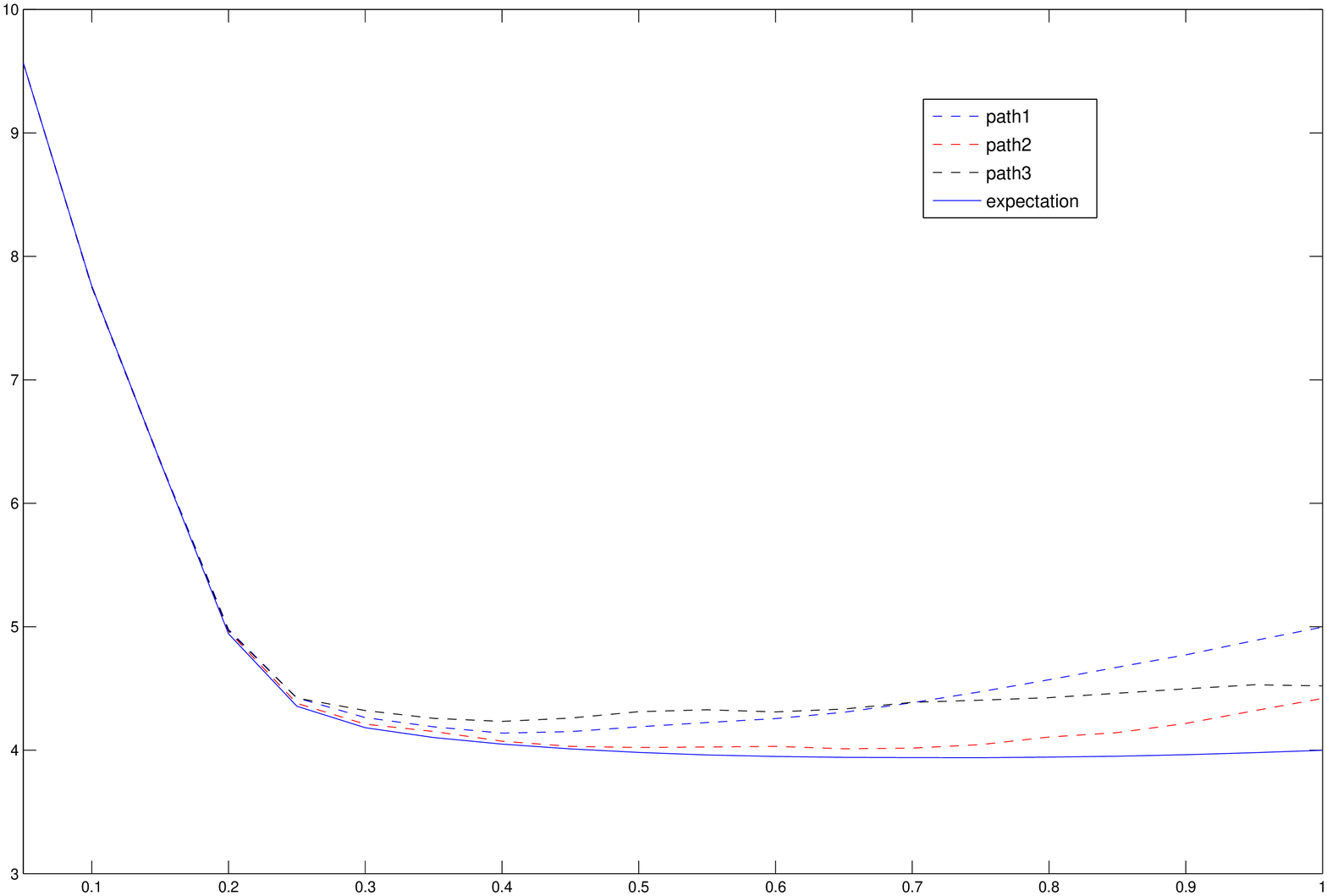}
\caption{Plot of $t\mapsto\cE(t)$, expectation and three individual
paths}\label{fig:Etotal}
\end{center}
\end{figure}
\section{Appendix}\label{sec:app}
For the reader's convenience we will recall the following lemmas proved in~\cite{GoldysLeTran2016}.
\begin{lemma}\label{lem:4.0}
For any real constants $\lambda_1$ and $\lambda_2$ with
$\lambda_1\not=0$, if $\vecpsi, \veczeta\in\R^3$
satisfy $|\veczeta|=1$, then there exists
$\vecvarphi\in\R^3$ satisfying
\begin{equation}\label{equ:app}
\lambda_1\vecvarphi
+
\lambda_2\vecvarphi\times\veczeta
=\vecpsi.
\end{equation}
As a consequence, if 
$\veczeta\in H^1\big((0,T);\mH^1(D)\big)$ with $|\veczeta(t,x)|=1$ a.e. in
$D_T$ and $\vecpsi\in L^2\big((0,T);W^{1,\infty}(D)\big)$, then 
$\vecvarphi\in L^2\big((0,T);\mH^1(D)\big)$.
\end{lemma}
\begin{lemma}\label{lem:Ih vh}
For any $\vecv\in\C(D)$, $\vecv_h\in\mV_h$ and
$\vecpsi\in\C_0^\infty(D_T)$ there hold
\begin{align*}
\norm{I_{\mV_h}\vecv}{\mL^{\infty}(D)}
&\le
\norm{\vecv}{\mL^{\infty}(D)}, \\
\norm{\vecm_{h,k}^-\times\vecpsi
-
I_{\mV_h}(\vecm_{h,k}^-\times\vecpsi)}{\mL([0,T],\mH^1(D))}^2
&\le
ch^2
\norm{\vecm_{h,k}^-}{\mL([0,T],\mH^1(D))}^2
\norm{\vecpsi}{\mW^{2,\infty}(D_T)}^2,
\end{align*}
where $\vecm_{h,k}^-$ is  defined in Defintion~\ref{def:mhk}
\end{lemma}

The next lemma defines a discrete $\mL^p$-norm in
$\mV_h$ which is equivalent to the usual $\mL^p$-norm.
\begin{lemma}\label{lem:nor equ}
There exist $h$-independent positive constants $C_1$ and
$C_2$ such that for all $p\in[1,\infty)$ and
$\vecu\in\mV_h$ there holds
\begin{equation*}
C_1\|\vecu\|^p_{\mL^p(\Omega)}
\leq
h^d
\sum_{n=1}^N |\vecu(\vecx_n)|^p
\leq
C_2\|\vecu\|^p_{\mL^p(\Omega)},
\end{equation*}
where $\Omega\subset\R^d$, d=1,2,3.
\end{lemma}
\section*{Acknowledgements}
The authors acknowledge financial support through the ARC projects DP160101755 and DP120101886.

\end{document}